\documentclass[10pt]{article}
\usepackage[latin1]{inputenc}
\usepackage[T1]{fontenc}

\usepackage[bitstream-charter]{mathdesign}

\usepackage{color}
\usepackage[usenames,dvipsnames,svgnames,table]{xcolor}
\usepackage{amsmath}
\usepackage{amsthm}
\swapnumbers
\usepackage[alphabetic]{amsrefs}
\usepackage[all]{xy}
\usepackage{mathtools}
\usepackage{stmaryrd}
\usepackage{enumitem}
\usepackage{fixltx2e}
\usepackage[defaultlines=3,all]{nowidow}
\usepackage{tabularx}
\usepackage{caption}

\usepackage{geometry}

\usepackage{microtype}

\theoremstyle{plain}
\newtheorem{thm}{Theorem}[section]
\newtheorem{lma}[thm]{Lemma}
\newtheorem{cor}[thm]{Corollary}
\newtheorem{prop}[thm]{Proposition}
\newtheorem*{announcethm}{Theorem}
\newtheorem*{announceprop}{Proposition}

\theoremstyle{definition}
\newtheorem{dfn}[thm]{Definition}

\theoremstyle{remark}
\newtheorem{rem}[thm]{Remark}
\newtheorem{ex}[thm]{Example}
\newtheorem{conv}[thm]{Convention}
\newtheorem{notn}[thm]{Notation}

\DeclareMathOperator{\Chow}{CH}

\DeclareMathOperator{\codim}{codim}

\DeclareMathOperator{\Cyc}{Z}
\DeclareMathOperator{\Db}{D^b}

\DeclareMathOperator{\Dperf}{D^{perf}}

\DeclareMathOperator{\Hom}{Hom}
\DeclareMathOperator{\id}{id}
\DeclareMathOperator{\im}{im}

\DeclareMathOperator{\Kzero}{K_0}

\DeclareMathOperator{\Spc}{Spc}
\DeclareMathOperator{\stab}{\!-stab}
\DeclareMathOperator{\supp}{supp}

\begin{document}
\title{Relative tensor triangular Chow groups, singular varieties and localization}
\author{Sebastian Klein}
\date{}

\maketitle

\begin{abstract}
	We extend the scope of Balmer's tensor triangular Chow groups to compactly generated triangulated categories $\mathcal{K}$ that only admit an action by a compactly-rigidly generated tensor triangulated category $\mathcal{T}$ as opposed to having a compatible monoidal structure themselves. The additional flexibility allows us to recover the Chow groups of a possibly singular algebraic variety $X$ from the homotopy category of quasi-coherent injective sheaves on $X$. We are also able to construct localization sequences associated to restricting to an open subset of $\Spc(\mathcal{T}^c)$, the Balmer spectrum of the subcategory of compact objects~$\mathcal{T}^c \subset \mathcal{T}$. This should be viewed in analogy to the exact sequences for the cycle and Chow groups of an algebraic variety associated to the restriction to an open subset.
\end{abstract}

\setcounter{tocdepth}{1}
\tableofcontents

\section{Introduction}
Let $\mathcal{T}$ be a tensor triangulated category, i.e.\ a triangulated category with a compatible symmetric monoidal structure, and assume that $\mathcal{T}$ is essentially small. The article \cite{balmerchow} introduced, for $p \in \mathbb{Z}$ a concept of cycle groups $\Cyc^{\Delta}_p(\mathcal{T})$ and Chow groups $\Chow^{\Delta}_p(\mathcal{T})$ for these categories, defined in analogy to a $\mathrm{K}$-theoretic description of the usual cycle groups and Chow groups of an algebraic variety (from \cite{quillenhigher}). In \cite{kleinchow} it was shown that these tensor triangular Chow groups recover the usual Chow groups of a non-singular algebraic variety $X$ when $\mathcal{T} = \Dperf(X)$, the derived category of perfect complexes on $X$, and that they possess reasonable functoriality properties. Furthermore, in \cite{kleinintersection}, it is demonstrated that under specific circumstances, one can even define an intersection product on the groups $\Chow^{\Delta}_p(\mathcal{T})$.

The aim of the article at hand is to extend the flexibility of the approach in two directions: so far, we do not have a notion of tensor triangular Chow groups for triangulated categories that do admit a compatible symmetric monoidal structure but are not essentially small, e.g.\ the (unbounded) derived category of quasi-coherent sheaves $\mathrm{D}(X)$ of a noetherian scheme. Furthermore, we might ask whether we can replace a symmetric monoidal structure by something weaker. In order to address these issues, we consider tensor triangulated categories $\mathcal{T}$ which are compactly-rigidly generated (see Definition \ref{dfncompriggen}) and have an action on a triangulated category $\mathcal{K}$ (which need not have a monoidal structure) in the sense of \cite{stevenson}. The action of $\mathcal{T}$ on $\mathcal{K}$ gives rise to an associated support theory for objects of $\mathcal{K}$ which we use to define \emph{relative tensor triangular cycle groups} and \emph{relative tensor triangular Chow groups}
\[\Cyc^{\Delta}_p(\mathcal{T}, \mathcal{K}) \quad \text{and} \quad \Chow^{\Delta}_p(\mathcal{T}, \mathcal{K})\]
for $p \in \mathbb{Z}$ (see Definition \ref{defcycchow}). When we set $\mathcal{K} = \mathcal{T}$ and let $\mathcal{T}$ act via its own monoidal structure, the new definition recovers the old one one in the sense that $\Cyc^{\Delta}_p(\mathcal{T}, \mathcal{T}) = \Cyc^{\Delta}_p(\mathcal{T}^c) $ and $\Chow^{\Delta}_p(\mathcal{T}, \mathcal{T}) = \Chow^{\Delta}_p(\mathcal{T}^c)$, where $\mathcal{T}^c \subset{\mathcal{T}}$ denotes the subcategory of compact objects (see Proposition \ref{relativecompactcomparison}). In particular, the groups $\Chow^{\Delta}_p(\mathrm{D}(X), \mathrm{D}(X))$ recover the usual Chow groups for a non-singular algebraic variety $X$ (see Corollary \ref{correlagr}).

However, our newly gained flexibility also allows us to recover the cycle groups and Chow groups of a \emph{possibly singular} algebraic variety. In Section \ref{sectionkinj} we consider the action of the derived category $\mathrm{D}(X)$ of a noetherian scheme on the category $\mathrm{K}(\mathrm{Inj}(X))$, the homotopy category of the category of quasi-coherent injective sheaves on $X$ (see Proposition \ref{propdxactsonkinjx}). This enables us to prove the following result:
\begin{announcethm}[\ref{thmsingagreement}]
	Let $X$ be a separated scheme of finite type over a field. Endow $\Spc(\mathrm{D}(X)^c) = X$ with the opposite of the Krull codimension as a dimension function. Then for all $p \leq 0$, we have isomorphisms
	\begin{align*}
	\Cyc^{\Delta}_{p}(\mathrm{D}(X),\mathrm{K}(\mathrm{Inj}(X))) &\cong \Cyc_{-p}(X) \\
	\Chow^{\Delta}_{p}(\mathrm{D}(X),\mathrm{K}(\mathrm{Inj}(X))) &\cong \Chow_{-p}(X)~.
	\end{align*}
\end{announcethm}

The last part of the article focuses on establishing localization sequences, in the following sense: the definitions of $\Cyc^{\Delta}_p(\mathcal{T}, \mathcal{K})$ and $\Chow^{\Delta}_p(\mathcal{T}, \mathcal{K})$ make use of Balmer's spectrum $\Spc(\mathcal{T}^c)$ of the essentially small tensor triangulated category $\mathcal{T}^c$ (see \cite{balmer2005spectrum}). If $U \subset \Spc(\mathcal{T}^c)$ is an open subset, there are associated Verdier quotients $\mathcal{T}_U$ of $\mathcal{T}$ and $\mathcal{K}_U$ of~$\mathcal{K}$. The triangulated category $\mathcal{T}_U$ remains compactly-rigidly generated and there is an induced action of $\mathcal{T}_U$ on $\mathcal{K}_U$ (see Lemma \ref{lmaresaction}). Furthermore, if $Z$ denotes the closed complement of $U$, we have a triangulated subcategory $\mathcal{K}_Z \subset \mathcal{K}$ to which the action of $\mathcal{T}$ restricts.  For example, if $X$ is a noetherian separated scheme, $U \subset X$ an affine open, $\mathcal{T} = D(X)$ and $\mathcal{K} = \mathrm{K}(\mathrm{Inj}(X))$, then $\Spc(\mathcal{T}^c) \cong X, \mathcal{T}_U = D(U)$ and $\mathcal{K}_U= \mathrm{K}(\mathrm{Inj}(U))$ (see Corollary \ref{corinjres}). In this case, the subcategory $\mathcal{K}_Z$ consists of all those complexes of $\mathrm{K}(\mathrm{Inj}(X))$ that have homological support contained in~$Z$. Our main result for this situation is the following:
\begin{announcethm}[\ref{thmlocseq}]
		Let $\mathcal{T}$ be a compactly-rigidly generated tensor triangulated category  acting on a triangulated category $\mathcal{K}$ such that the action satisfies the local-to-global principle. Let $U \subset \Spc(\mathcal{T}^c)$ be an open subset with closed complement $Z$, denote by $\mathcal{K}_Z \subset \mathcal{K}$ the full subcategory of those objects $a$ with $\supp(a) \subset Z$ and let $\mathcal{K}_U := \mathcal{K}/\mathcal{K}_Z$. Then the quotient functor $\mathcal{K} \to \mathcal{K}_U$ and the inclusion $\mathcal{K}_Z \to \mathcal{K}$ induce exact localization sequences
		\[0 \to \Cyc^{\Delta}_p(\mathcal{T},\mathcal{K}_Z) \xrightarrow{i_p} \Cyc^{\Delta}_p(\mathcal{T}, \mathcal{K}) \xrightarrow{l_p} \Cyc^{\Delta}_p(\mathcal{T}_U, \mathcal{K}_U) \to 0\]
		and
		\[\Chow^{\Delta}_p(\mathcal{T}, \mathcal{K}) \xrightarrow{\ell_p} \Chow^{\Delta}_p(\mathcal{T}_U, \mathcal{K}_U) \to 0\]
		for all $p \in \mathbb{Z}$. Furthermore, $i_P$ induces a map $\iota_p: \Chow^{\Delta}_p(\mathcal{T}, \mathcal{K}_Z) \to \Chow^{\Delta}_p(\mathcal{T}, \mathcal{K})$ such that $\im(\iota_p) \subset \ker(\ell_p)$.
\end{announcethm}
The above exact sequences should be compared to the corresponding ones for cycle groups and Chow groups of algebraic varieties (see \cite{fulton} and Remark \ref{remvarexseq}). At the moment the author does not know whether we have $\im(\iota_p) = \ker(\ell_p)$ in this generality. However, we are able to give a criterion for when the map $\ell_p$ is an isomorphism:
\begin{announceprop}[\ref{propchowresiso}]
	Assume that $p > \dim(Z)$ and that $\mathcal{K}^c/\mathcal{K}^c_Z$ is idempotent complete. Then 
	\[\ell_p: \Chow^{\Delta}_p(\mathcal{T}, \mathcal{K}) \to \Chow^{\Delta}_p(\mathcal{T}_U, \mathcal{K}_U) \]
	is an isomorphism.
\end{announceprop}

\textbf{Acknowledgments:} The results in this paper are partially based on and contained in the author's Ph.D.\ thesis, which was written at Utrecht University and jointly supervised by Paul Balmer and Gunther Cornelissen. The author would like to thank both of them for their support. The members of the thesis committee also pointed out some corrections and suggested some improvements, for which the author is grateful. Furthermore, the author would like to thank Greg Stevenson and Pieter Belmans for helpful discussions. The author acknowledges the support of the European Union for the ERC grant No.\ 257004-HHNcdMir.

\section{Preliminaries}
\label{sectionprelim}

\subsection{Tensor triangular Chow groups for essentially small tensor triangulated categories.}
We first summarize the parts of the theory of Chow groups for essentially small tensor triangulated categories which are are relevant for us. 

\begin{dfn}
An essentially small triangulated category $\mathcal{C}$ is called \emph{tensor triangulated} if it comes equipped with a compatible symmetric monoidal structure $\otimes:  \mathcal{C} \times \mathcal{C} \to \mathcal{C}$, i.e.\ if the functor $\otimes$ is exact in both variables and the two isomorphisms
\[\Sigma a \otimes \Sigma b \cong \Sigma(a \otimes \Sigma b) \cong \Sigma^2(a \otimes b) \quad \text{and} \quad \Sigma a \otimes \Sigma b \cong \Sigma(\Sigma a \otimes b) \cong \Sigma^2(a \otimes b)\]
differ by a sign $-1$ at most. An essentially small tensor triangulated category is called \emph{rigid} if there is an exact functor $D: \mathcal{C}^{\mathrm{op}} \to \mathcal{C}$ and a natural isomorphism $\Hom_{\mathcal{C}}(a \otimes b, c) \cong \Hom_{\mathcal{C}}(b, D(a) \otimes c)$ for all objects $a,b,c \in \mathcal{C}$. The object $D(a)$ is called the \emph{dual} of $a$.
\label{dfnttcat}
\end{dfn}

Recall from \cite{balmer2005spectrum} that for such a category, we have a notion of \emph{spectrum}, a spectral topological space $\Spc(\mathcal{C})$ such that every object $a \in \mathcal{C}$ has a \emph{support} $\supp(a) \subset \Spc(\mathcal{C})$, a closed subset. In the case $\mathcal{C} = \Dperf(X)$, the derived category of perfect complexes on a quasi-compact, quasi-separated scheme $X$, one obtains $\Spc(\mathcal{C}) \cong X$ and $\supp(a^{\bullet}) = \mathrm{supph}(a^{\bullet})$ under this isomorphism, where $\mathrm{supph}(a^{\bullet})$ is the homological support of the complex of sheaves $a^{\bullet}$. If $\mathcal{C} = kG\stab$, the stable category of finitely generated left $kG$-modules for a finite group $G$ and a field $k$, we have $\Spc(\mathcal{C}) = \mathcal{V}_G(k)$, the projective support variety of $G$ over $k$, and $\supp(M) = \mathcal{V}_G(M)$ under this isomorphism, the variety of the module $M$.

\begin{dfn}[see \cite{balmerfiltrations}*{Definition 3.1}] 
	A \emph{dimension function} on $\mathcal{C}$ is a map 
	\[\dim: \Spc(\mathcal{C}) \to \mathbb{Z} \cup \lbrace \pm \infty \rbrace\]
	such that the following two conditions hold:
	\begin{enumerate}
		\item If $Q, P$ are points of $\Spc(\mathcal{C})$ such that $Q \in \overline{\lbrace P \rbrace}$, then~$\dim(Q) \leq \dim(P)$.
		\item If $Q \in \overline{\lbrace P \rbrace}$ and $\dim(Q) = \dim(P) \in \mathbb{Z}$, then~$Q = P$.
	\end{enumerate}
	For a subset $V \subset \Spc(\mathcal{C})$, we define $\dim(V) := \sup \lbrace \dim(P) | P \in V \rbrace$. For every $p \in \mathbb{Z} \cup \lbrace \pm \infty \rbrace$, we define the full subcategory
	\[\mathcal{C}_{(p)} := \lbrace a \in \mathcal{C} : \dim(\supp(a)) \leq p \rbrace ~.\]
\end{dfn}

\begin{ex}
	Two examples of dimension functions are given by the \emph{Krull dimension} $\dim_{\mathrm{Krull}}$ and the \emph{opposite of the Krull codimension} $-\codim_{\mathrm{Krull}}$, which are defined in the usual manner (see \cite{balmerfiltrations}*{Examples 3.2, 3.3}). Indeed, this is sensible since $\Spc(\mathcal{C})$ is always a spectral topological space, i.e.\ it is homeomorphic to the prime ideal spectrum of some (usually unknown, possibly horrible) commutative ring.
\end{ex}

We will now explain a notion of Chow groups for tensor triangulated categories that was introduced in \cite{balmerchow} and further studied in \cite{kleinchow}.
\begin{dfn}
	An additive category $\mathcal{A}$ is called \emph{idempotent complete} if all idempotent endomorphisms in $\mathcal{A}$ split. Every triangulated category $\mathcal{C}$ admits an exact embedding into its \emph{idempotent completion} $J: \mathcal{C} \hookrightarrow \mathcal{C}^{\natural}$, where $\mathcal{C}^{\natural}$ is an idempotent complete triangulated category such that every exact functor $\mathcal{C} \to \mathcal{D}$ with $\mathcal{D}$ triangulated and idempotent complete factors uniquely through $J$ (see \cite{balschlichidem}).
\end{dfn}

Fix a dimension function $\dim$ on $\mathcal{C}$ and consider the diagram
\[\xymatrix{
	\mathcal{C}_{(p)} \ar@{^{(}->}[r]^I \ar@{->>}[d]^Q& \mathcal{C}_{(p+1)} \\
	\mathcal{C}_{(p)}/\mathcal{C}_{(p-1)} \ar@{^{(}->}[r]^-J & (\mathcal{C}_{(p)}/\mathcal{C}_{(p-1)})^{\natural}
}
\]
where $I,J$ denote the obvious embeddings and $Q$ is the Verdier quotient functor. After applying $\Kzero$ we get a diagram
\[
\xymatrix{
	\Kzero(\mathcal{C}_{(p)}) \ar[r]^i \ar@{->>}[d]^q& \Kzero(\mathcal{C}_{(p+1)}) \\
	\Kzero(\mathcal{C}_{(p)}/\mathcal{C}_{(p-1)}) \ar@{^{(}->}[r]^-{j} & \Kzero\left( (\mathcal{C}_{(p)}/\mathcal{C}_{(p-1)})^{\natural} \right) 
}
\]
where the lowercase maps are induced by the uppercase functors.

\begin{dfn}[see \cite{balmerchow}*{Definitions 8 and 10}]
	The \emph{$p$-dimensional cycle group of $\mathcal{C}$} is defined as
	\[\Cyc^{\Delta}_{p}(\mathcal{C}) := \Kzero\left( (\mathcal{C}_{(p)}/\mathcal{C}_{(p-1)} )^{\natural}\right)~.\]
	The \emph{$p$-dimensional Chow group of $\mathcal{C}$} is defined as
	\[\Chow^{\Delta}_{p}(\mathcal{C}) := \Cyc^{\Delta}_{p}(\mathcal{C}) / j \circ q (\ker(i))~.\]
	\label{defcycchow}
\end{dfn}

Let us recall from \cite{kleinchow} the following theorem that justifies the name "Chow groups":
\begin{thm}[see \cite{kleinchow}*{Theorem 3.2.6}]
	Let $X$ be a separated, non-singular scheme of finite type over a field and assume that the tensor triangulated category $\Dperf(X)$ is equipped with the dimension function~$-\codim_{\mathrm{Krull}}$. Then there are isomorphisms
	\[\Cyc^{\Delta}_{p}\left(\Dperf(X)\right) \cong Z^{-p}(X) \qquad \text{and} \qquad \Chow^{\Delta}_{p}\left(\Dperf(X)\right) \cong \Chow^{-p}(X) \]
	for all $p \in \mathbb{Z}$.
	\label{agreementthm}
\end{thm}

\subsection{Actions of tensor triangulated categories}
We want to extend the scope of tensor triangular cycle groups and Chow groups to categories which need not be essentially small or even have a monoidal structure. Here, we lay the foundations for Section \ref{sectionrelchow}.
\begin{dfn}
	A triangulated category $\mathcal{T}$ is called a \emph{compactly-rigidly generated tensor triangulated category} if
	\begin{enumerate}[label=(\roman*)]
		\item \emph{$\mathcal{T}$ is a compactly generated triangulated category.} We implicitly assume here that $\mathcal{T}$ has set-indexed coproducts. Note that this implies that $\mathcal{T}$ is not essentially small.
		\item \emph{$\mathcal{T}$ is equipped with a compatible closed symmetric monoidal structure
			\[\otimes: \mathcal{T} \times \mathcal{T} \to \mathcal{T}\]
		with unit object $\mathbb{I}$}. Here, a symmetric monoidal structure on $\mathcal{T}$ is \emph{closed} if for all objects $a \in \mathcal{T}$ the functor $a \otimes -$ has a right adjoint $\underline{\mathrm{hom}}(a,-)$.
		\item \emph{$\mathbb{I}$ is compact and all compact objects of $\mathcal{T}$ are rigid.} Let $\mathcal{T}^c \subset \mathcal{T}$ denote the full subcategory of compact objects of $\mathcal{T}$. Then we require that $\mathbb{I} \in \mathcal{T}^c$ and that all objects $a$ of $\mathcal{T}^c$ are rigid, i.e.\ for every object $b \in \mathcal{T}$ the natural map
		\[\underline{\circ}: \underline{\mathrm{hom}}(a,\mathbb{I}) \otimes b \cong \underline{\mathrm{hom}}(a,\mathbb{I}) \otimes  \underline{\mathrm{hom}}(\mathbb{I},b) \to \underline{\mathrm{hom}}(a,b)~,\]
		is an isomorphism.
	\end{enumerate}
	\label{dfncompriggen}
\end{dfn}

\begin{rem}[See \cite{balfavibig}*{Hypotheses 1.1}]
Note that with this definition, $\mathcal{T}$ is not essentially small since it has set-indexed coproducts. Hence it cannot be a tensor triangulated category in the sense of Definition \ref{dfnttcat}. However, the triangulated subcategory of compact objects $\mathcal{T}^c \subset \mathcal{T}$ is of this type: our assumptions guarantee that the full subcategories of compact and rigid objects coincide (see \cite{hovpalmstrick}*{Theorem 2.1.3(d)}) and that the tensor product restricts to $\mathcal{T}^c$. Furthermore $\mathcal{T}^c$ is rigid: the functor $D$ of Definition \ref{dfnttcat} is given by $\underline{\mathrm{hom}}(-,\mathbb{I})$.
\end{rem}

\begin{conv}
	Throughout, we assume that $\mathcal{T}$ is a compactly-rigidly generated tensor triangulated category that satisfies the following conditions:
	\begin{enumerate}[label=(\roman*)]
		\item \emph{The essentially small tensor triangulated category $\mathcal{T}^c$ is equipped with a dimension function $\dim$ and $\Spc(\mathcal{T}^c)$ is a noetherian topological space}.
		\item \emph{$\mathcal{T}$ acts on a (fixed) compactly generated triangulated category $\mathcal{K}$ via an action $\ast$ in the sense of~\cite{stevenson}.} 
	\end{enumerate}
	\label{convbigcond}
\end{conv}

Let us briefly recall from \cite{stevenson}*{Definition 3.2} what it means for $\mathcal{T}$ to have an action $\ast$ on $\mathcal{K}$.\index{Action of a tensor triangulated category}
We are given a biexact bifunctor
\[\ast: \mathcal{T} \times \mathcal{K} \to \mathcal{K}\]
that commutes with coproducts in both variables, whenever they exist. Furthermore we are given natural isomorphisms
\begin{align*}
\alpha_{X,Y,A} &: (X \otimes Y) \ast A \overset{\sim}{\longrightarrow} X \ast (Y \ast A)\\
l_A &: \mathbb{I} \ast A \overset{\sim}{\longrightarrow} A
\end{align*}
for all objects $X,Y \in \mathcal{T}, A \in \mathcal{K}$. These natural isomorphisms should satisfy a list of natural coherence conditions which we omit here and refer the reader to \cite{stevenson}*{Definition 3.2}. 

\begin{rem}
	The assumption on the noetherianity of $\Spc(\mathcal{T}^c)$ simplifies some aspects of the sequel. For example, we do not need to distinguish between \emph{Thomason subsets} of $\Spc(\mathcal{T}^c)$ (unions of closed subsets with quasi-compact complements) and specialization-closed subsets. Furthermore, all points of $\Spc(\mathcal{T}^c)$ are \emph{visible} under this hypothesis, meaning there exist for each point $x \in \Spc(\mathcal{T}^c)$ Thomason subsets $V,W$ such that $\lbrace x \rbrace = V \cap W$ (see \cite{stevenson}*{Section 5}).
\end{rem}

\begin{rem}
	With this definition, a compactly-rigidly generated tensor triangulated category $\mathcal{T}$ as in Definition \ref{dfncompriggen} has an action on itself via $\otimes$. The natural isomorphisms $\alpha_{X,Y,A}, l_{A}$ are then given by the associator and unitor of the symmetric monoidal structure on $\mathcal{T}$ and one checks that all the required coherence conditions hold as the coherence conditions for the monoidal structure on $\mathcal{T}$ are satsified and the bifunctor $\otimes$ is compatible with the triangulated structure on $\mathcal{T}$. The functor $\otimes$ always preserves coproducts in both variables for any closed symmetric monoidal structure on any category, since the functor $a \otimes -$ has a right adjoint for all objects $a \in \mathcal{T}$ (see e.g.\ \cite{hovpalmstrick}*{Remark A.2.2}).
	\label{remselfaction}
\end{rem}

\begin{dfn}
	For a triangulated category $\mathcal{C}$, a triangulated subcategory $\mathcal{L} \subset \mathcal{C}$ is called \emph{localizing} if it is closed under taking set-indexed coproducts (in $\mathcal{C}$). A thick triangulated subcategory $\mathcal{L} \subset \mathcal{C}$ is called \emph{Bousfield} if the Verdier quotient functor $\mathcal{T} \to \mathcal{T}/\mathcal{S}$ has a right-adjoint. A Bousfield subcategory $\mathcal{L} \subset \mathcal{C}$ is called \emph{smashing} if the right-adjoint of the Verdier quotient functor $\mathcal{T} \to \mathcal{T}/\mathcal{S}$ preserves set-indexed coproducts.
\end{dfn}

Let us summarize some well-known facts about localizing, Bousfield and smashing subcategories: a Bousfield subcategory is necessarily localizing (see \cite{neemantc}*{Remark 9.1.17}). If $\mathcal{L} \subset \mathcal{C}$ is compactly generated, then the Brown representability theorem implies that the inclusion functor has a right-adjoint, which is equivalent to $\mathcal{L}$ being a Bousfield subcategory (see \cite{neemantc}*{Proposition 9.1.18}). If $\rho$ denotes the right-adjoint of the Verdier quotient $Q_{\mathcal{L}}: \mathcal{C} \to \mathcal{C}/\mathcal{L}$, the functor $L_{\mathcal{L}} := \rho \circ Q$ is called the \emph{(Bousfield) localization functor associated to $\mathcal{L}$} and it induces an equivalence~$\mathcal{C}/\mathcal{L} \xrightarrow{\sim} \im(L)$. Given a Bousfield subcategory $\mathcal{L}$, every object $a \in \mathcal{C}$ fits into a distinguished  \emph{localization triangle}
\[\Gamma_{\mathcal{L}}(a) \to a \to L_{\mathcal{L}} (a) \to \Sigma \Gamma_{\mathcal{L}}(a)\]
where the \emph{acyclization functor} $\Gamma_{\mathcal{L}}$ is given as the composition of the inclusion $\mathcal{L} \hookrightarrow \mathcal{C}$ and its right-adjoint. This triangle is unique up to unique isomorphism among all triangles of the form
\[x \to a \to y \to \Sigma x\]
with $x \in \mathcal{L}$ and $y \in {^{\perp}\mathcal{L}} := \lbrace b \in \mathcal{C} : \Hom_{\mathcal{C}}(l,b) = 0 ~\forall l \in \mathcal{L} \rbrace$. If $\mathcal{T}$ is as in Definition \ref{dfncompriggen}, and $\mathcal{L} \subset \mathcal{T}$ is a smashing subcategory, it can be shown that the localization functor $L_{\mathcal{L}}$ is given by $L_{\mathcal{L}}(\mathbb{I}) \otimes -$ and that $\Gamma_{\mathcal{L}}$ is given as~$\Gamma_{\mathcal{L}}(\mathbb{I}) \otimes -$. Furthermore, if $\mathcal{J} \subset \mathcal{T}^c$ is a thick triangulated subcategory such that $\mathcal{T}^c \otimes \mathcal{J} = \mathcal{J}$ (i.e.\ $\mathcal{J}$ is a \emph{thick tensor ideal}) then $\langle \mathcal{J} \rangle$, the smallest localizing subcategory of $\mathcal{T}$ containing $\mathcal{J}$, is a smashing subcategory.

Following \cite{balfavibig}, we can assign to every object $a \in \mathcal{T}$ a support~$\supp(a) \subset \Spc(\mathcal{T}^c)$: given a specialization-closed subset $V \subset \Spc(\mathcal{T}^c)$, we have a distinguished triangle
\[\Gamma_V (\mathbb{I}) \to \mathbb{I} \to  L_V (\mathbb{I}) \to \Sigma \Gamma_V (\mathbb{I}) \]
where $\Gamma_V$ and $L_V$\label{locacycdefpage} are the acyclization and localization functors associated to the smashing ideal that is generated by the thick $\otimes$-ideal of objects in $\mathcal{T}^c$ with support contained in~$V$.
For objects $a \in \mathcal{K}$, we set $\Gamma_V a := \Gamma_V (\mathbb{I}) \ast a$ and~$L_V a := L_V (\mathbb{I}) \ast a$. By \cite{stevenson}*{Lemma 4.4} $L_V(\mathbb{I}) \ast -$ is a Bousfield localization functor on $\mathcal{K}$, associated to the Bousfield subcategory $\Gamma_V (\mathbb{I}) \ast \mathcal{K}$, and for each object $a \in \mathcal{K}$ we obtain localization triangles
\[\Gamma_V a \to a \to L_V a \to \Sigma \Gamma_V a~.\]
Since the functor $L_V(\mathbb{I}) \ast -$ preserves coproducts by our assumptions on the action $\ast$, it follows from \cite{krauseloc}*{Proposition 5.5.1} that the subcategory $\Gamma_V (\mathbb{I}) \ast \mathcal{K}$ is even smashing.

For a point $x \in \Spc(\mathcal{T}^c)$ the subsets $\overline{\lbrace x \rbrace}$ and $Y_x := \lbrace P \in \Spc(\mathcal{T}^c) : x \notin \overline{\lbrace P \rbrace} \rbrace$\label{Yxdefpage} are special\-ization-closed and so we define the ``residue object'' $\Gamma_x \mathbb{I} \in \mathcal{T}$ as $\Gamma_{\overline{\lbrace x \rbrace}} L_{Y_x} (\mathbb{I})$. For $a \in \mathcal{T}$, we now define the \emph{support of an object $a \in \mathcal{T}$} as
\[\supp(a) := \lbrace x \in \Spc(\mathcal{T}^c) | \Gamma_x \mathbb{I} \otimes a \neq 0 \rbrace.\] 
In \cite{stevenson}, the same residue objects are used to define supports for objects of $\mathcal{K}$. For an object $b \in \mathcal{K}$, set $\Gamma_x b := \Gamma_x \mathbb{I} \ast b$\label{Gammaxdefpage}, then we define the \emph{support of $b$} as
\[\supp(b) := \lbrace x \in \Spc(\mathcal{T}^c) | \Gamma_x b \neq 0 \rbrace.\] 

This notion of support gives us a way to describe certain subcategories of $\mathcal{K}$. A triangulated subcategory $\mathcal{M} \subset \mathcal{K}$ is called \emph{localizing $\mathcal{T}$-submodule} if it is a localizing subcategory such that $\mathcal{T} \ast \mathcal{M} \subset \mathcal{M}$. We obtain order-preserving maps
\[
	\xymatrix@R=0.5em{
		\lbrace \text{subsets of } \Spc(\mathcal{T}^c) \rbrace \ar@<0.5ex>[r]^-{\tau_{\mathcal{K}}} &  \lbrace \text{localizing $\mathcal{T}$-submodules of } \mathcal{K}\rbrace  \ar@<0.5ex>[l]^-{\sigma_{\mathcal{K}}}\\
		 S \ar@{|->}[r]& \lbrace t \in \mathcal{K}: \supp(t) \subset S \rbrace \\
	\bigcup_{t \in \mathcal{M}} \supp(t) & \mathcal{M} \ar@{|->}[l]
	}
\]
where the ordering on both sides is given by inclusion (see \cite{stevenson}*{Definition 5.9}).

We record the following properties of the support that will be very useful for the sequel.
\begin{prop}[cf. \cite{stevenson}*{Proposition 5.7}]
	Let $V \subset \Spc(\mathcal{T}^c)$ be a subset closed under specialization and $a$ be an object of $\mathcal{K}$. Then
	\[\supp(\Gamma_V  a) = \supp(a) \cap V\]
	and
	\[\supp(L_V a) = \supp(a) \cap (\Spc(\mathcal{T}^c) \setminus V)~.\]
	\label{propsuppprops}
\end{prop}

\begin{rem}
	In \cite{stevenson}*{Proposition 5.7}, Proposition \ref{propsuppprops} is proved for those specialization-closed subsets $V$ which are contained in the subset $\mathrm{Vis}(\mathcal{T}^c) \subset \Spc(\mathcal{T}^c)$ of so-called \emph{visible} points of the spectrum. The set $\mathrm{Vis}(\mathcal{T}^c)$ coincides with $\Spc(\mathcal{T})$ in our case, since we assumed the latter space to be noetherian.
\end{rem}

\begin{dfn}[cf. \cite{stevenson}*{Definition 6.1}] \index{Local-to-global principle}
	We say that the action $\ast$ of $\mathcal{T}$ on $\mathcal{K}$ satisfies the \emph{local-to-global principle} if for each $a$ in $\mathcal{K}$
	\[\langle a \rangle_{\ast} = \langle \Gamma_x a | x \in \Spc(\mathcal{T}^c) \rangle_{\ast}\]
	where for a collection of objects $S \subset \mathcal{K}$ we denote by $\langle S \rangle_{\ast}$ the smallest localizing $\mathcal{T}$-submodule of $\mathcal{K}$ containing~$S$.
	\label{locglobdef}
\end{dfn}
\begin{rem}
	The local-to-global principle holds very often, e.g.\ when $\mathcal{T}$ arises as the homotopy category of a monoidal model category (cf. \cite{stevenson}*{Proposition 6.8}). If the local-to-global principle holds, it implies the useful consequence that the relative support detects the vanishing of an object: if $\supp(a) = \emptyset$ then
	\[\langle a \rangle_{\ast} = \langle \Gamma_x a | x \in \Spc(\mathcal{T}^c) \rangle_{\ast} = \langle 0 \rangle_{\ast} = 0\]
	which implies that~$a = 0$.
	\label{locglobvanishing}
\end{rem}

For $p \in \mathbb{Z}$, let 
\[V_{\leq p} := \lbrace x \in \Spc(\mathcal{T}^c) | \dim(x) \leq p \rbrace, \quad V_p := \lbrace x \in \Spc(\mathcal{T}^c) | \dim(x) = p \rbrace \label{Vleqpdefpage}\] 
and set~$\Gamma_p A := \Gamma_{V_{\leq p}} L_{V_{\leq p-1}} A$\label{Gammapdefpage}. Let us state a decomposition theorem analogous to \cite{balmerfiltrations}.

\begin{notn}[cf.\ \cite{stevenfilt}*{Notation 3.6}]
	Let $\mathcal{L}_i$ be a collection of localizing subcategories of $\mathcal{K}$, indexed by a set $I$. Then $\prod_{i \in I} \mathcal{L}_i$ denotes the subcategory of $\mathcal{K}$ whose objects are coproducts of the objects of $\mathcal{L}_i$, where the morphisms and the triangulated structure are defined componentwise with respect to $I$.
	\label{notnprodcat}
\end{notn}

\begin{prop}[cf. \cite{stevenfilt}*{Lemma 4.3}] \index{Decomposition theorem}
	Suppose that the action of $\mathcal{T}$ on $\mathcal{K}$ satisfies the local-to-global principle and let $p \in \mathbb{Z}$. There is an equality of subcategories
	\[\Gamma_p \mathcal{K} = \tau_{\mathcal{K}}(V_p) = \prod_{x \in V_p} \Gamma_x \mathcal{K}\]
	where $\Gamma_x \mathcal{K}$ denotes the essential image of the functor $\Gamma_x(\mathbb{I}) \ast -$.
	\qed
	\label{stevensondecomposition}
\end{prop}

We can also describe the compact objects of $\Gamma_p \mathcal{K}$ in a similar fashion:
\begin{lma}
	Let $\mathcal{L}_i$ be a collection of localizing subcategories of $\mathcal{K}$, indexed by a set $I$, $\prod_{i \in I} \mathcal{L}_i$ as in \ref{notnprodcat}. Then
	\[\left(\prod_{i \in I} \mathcal{L}_i\right)^c = \coprod_{i \in I} \mathcal{L}_i^c~,\]
	where $\coprod_{i \in I} \mathcal{L}_i^c$ denotes the subcategory of $\mathcal{K}$ whose objects are finite coproducts of the objects of $\mathcal{L}_i^c$, where the morphisms and the triangulated structure are defined componentwise with respect to $I$. In particular,
	\[(\Gamma_p \mathcal{K})^c  = \coprod_{x \in V_p} (\Gamma_x \mathcal{K})^c~.\]
	\label{lmacompprod}
\end{lma}
\begin{proof}
	``$\supseteq$'': Let 
	\[a = \bigoplus_{\substack{i \in J \subset I \\ |J| <\infty}} a_i, \quad a_i \in \mathcal{L}_i^c\] 
	be an object of $\coprod_{i \in I} \mathcal{L}_i^c$ and $b^s, s \in S$ a collection of objects of $\prod_{i \in I} \mathcal{L}_i$ with
	\[b^s = \bigoplus_{i \in I} b^s_i, \quad b^s_i \in \mathcal{L}_i\]
	and 
	\[f: a \to \bigoplus_{s \in S} b^s\] 
	a morphism of $\prod_{i \in I} \mathcal{L}_i$. By definition, this means that $f$ is given by morphisms 
	$f_i: a_i \to \bigoplus_{s \in S} b^s_i$
	and since the $a_i$ are compact, each $f_i$ will factor through a finite coproduct of objects $b^s_i$. Furthermore, $f_i = 0$ for all but finitely many $i$ (namely those contained in $J$). This implies that $f$ factors through a finite coproduct of objects $b_s$. Hence, $a$ is compact.
	
	``$\subseteq$'': Assume that $a$ is an object of $\left(\prod_{i \in I} \mathcal{L}_i\right)^c$. By definition of the category $\prod_{i \in I} \mathcal{L}_i$, we can write $a = \bigoplus_{i \in I} a_i$ with $a_i \in \mathcal{L}_i$. We consider the identity morphism $a \xrightarrow{\id} \bigoplus_{i \in I} a_i$. Since $a$ is compact, it must factor through a finite coproduct of the $a_i$, w.l.o.g.~$a = \bigoplus_{i=1}^n a_i$.
	It remains to show that $a_i \in \mathcal{L}_i^c$ for~$i=1, \ldots, n$. Let $b^s, s \in S$ be a family of objects in $\mathcal{L}_i$ and let 
	\[f_i: a_i \to \bigoplus_{s \in S} b^s\]
	be a morphism in $\mathcal{L}_i$. This defines a morphism $f: a \to \bigoplus_{s \in S} b^s$ in $\prod_{i \in I} \mathcal{L}_i$ by setting $f_j = 0$ for $i \neq j$. Since $a$ is compact, it follows that $f$ must factor through a finite coproduct of objects $b^s$, which, by the definition of morphisms of $\prod_{i \in I} \mathcal{L}_i$, means that $f_i$ must do so. Hence, $a_i$ is compact.
\end{proof}

We give another description of $\Gamma_p \mathcal{K}$ that will elucidate its relation to Definition~\ref{defcycchow}. For $p \in \mathbb{Z}$, define 
\[\mathcal{K}_{(p)} := \tau_{\mathcal{K}}(V_{\leq p})~.\label{Kpdefpage}\]

\begin{lma}
	Assume the local-to-global principle holds for the action of $\mathcal{T}$ on~$\mathcal{K}$. Then, for all specialization-closed subsets $V \subset \Spc(\mathcal{T}^c)$, there is an equality of subcategories
	\[ \tau_{\mathcal{K}}(V) = \Gamma_{V} \mathcal{K} := \lbrace a \in \mathcal{K} | \exists a': a \cong \Gamma_{V} a' \rbrace~.\]
	In particular, $\mathcal{K}_{(p)} =  \Gamma_{V_{\leq p}} \mathcal{K}$.
	\label{subcatequality}
\end{lma}
\begin{proof}
	Let $a \in \Gamma_{V} \mathcal{K}$, then we have an isomorphism~$a \cong \Gamma_{V} a'$. By Proposition \ref{propsuppprops}, we know that $\supp(a) = \supp(\Gamma_{V} a')  = \supp(a') \cap V$, from which it follows that~$\supp(a) \subset V$. Thus,~$a \in \tau_{\mathcal{K}}(V)$. 
	
	Conversely, assume that~$a \in \tau_{\mathcal{K}}(V)$. Consider the localization triangle
	\[\Gamma_{V}  a \to a \to  L_{V} a \to \Sigma \Gamma_{V} a~.\]
	As $\supp(a) \subset V$, it follows again from Proposition \ref{propsuppprops} that $\supp(L_{V} a) = \emptyset$ and therefore $L_{V} a = 0$ by the local-to-global principle, as shown in Remark~\ref{locglobvanishing}. Therefore $\Gamma_{V} a \cong a$ which implies that~$a \in \Gamma_{V} \mathcal{K}$.
\end{proof}

\begin{lma}
	Let $V,W \subset \Spc(\mathcal{T}^c)$ be specialization-closed subsets. Then the localization functor $L_W(\mathbb{I}) \ast -$ induces an equivalence
	\[\tau_{\mathcal{K}}(V)/(\tau_{\mathcal{K}}(V) \cap \tau_{\mathcal{K}}(W)) \xrightarrow{\sim} L_W(\mathbb{I}) \ast(\tau_{\mathcal{K}}(V))~.\]
	\label{lmasubcaloc}
\end{lma}
\begin{proof}
	Let $a$ be an object of $\tau_{\mathcal{K}}(V)$, $b$ an object of $\tau_{\mathcal{K}}(W)$ and $f: a \to b$ a morphism in $\mathcal{K}$. Consider the commutative diagram 
	\[
		\xymatrix{
			\Gamma_V a \ar[r]\ar[d] & a \ar[r] \ar[d]^f & L_V a \ar[r] \ar[d] & \Sigma \Gamma_V a \ar[d] \\
			\Gamma_{V } b \ar[r] & b \ar[r]  & L_{V} b \ar[r] & \Sigma \Gamma_{V} b
		}
	\]
	obtained by the functoriality of localization triangles. It follows from Proposition \ref{propsuppprops} that $\Gamma_V a \cong a$ and hence, $f$ factors through $\Gamma_{V } b$, an object of  $\tau_{\mathcal{K}}(V) \cap \tau_{\mathcal{K}}(W)$ since $\supp (\Gamma_{V } b) \subset V \cap W$. From \cite{krauseloc}*{Lemma 4.7.1} we deduce that the inclusion $\tau_{\mathcal{K}}(V) \hookrightarrow \mathcal{K}$ induces a fully faithful functor
	\begin{equation}
	\tau_{\mathcal{K}}(V)/(\tau_{\mathcal{K}}(V) \cap \tau_{\mathcal{K}}(W)) \hookrightarrow \mathcal{K}/\tau_{\mathcal{K}}(W)~.
	\label{eqnindquotff}
	\end{equation}
	The essential image of this functors equals the essential image of the Verdier quotient $\mathcal{K} \to \mathcal{K}/\tau_{\mathcal{K}}(W)$ restricted to~$\tau_{\mathcal{K}}(V)$. Composing the functor (\ref{eqnindquotff}) with the right-adjoint of the Verdier quotient $\mathcal{K} \to \mathcal{K}/\tau_{\mathcal{K}}(W)$, we obtain a fully faithful functor $\tau_{\mathcal{K}}(V)/(\tau_{\mathcal{K}}(V) \cap \tau_{\mathcal{K}}(W)) \to \mathcal{K}$ and it remains to show that its essential image is~$L_W(\mathbb{I}) \ast (\tau_{\mathcal{K}}(V))$. But this is clear, since $L_W(\mathbb{I}) \ast -$ is by definition given as the composition of the Verdier quotient functor $\mathcal{K} \to \mathcal{K}/\tau_{\mathcal{K}}(W)$ and its right adjoint.
\end{proof}

The following statement is the desired description of~$\Gamma_p \mathcal{K}$.
\begin{lma}
	Suppose that the action of $\mathcal{T}$ on $\mathcal{K}$ satisfies the local-to-global principle and let $p \in \mathbb{Z}$. There is an equality of subcategories
	\[\Gamma_p \mathcal{K}  = \mathcal{K}_{(p)}/\mathcal{K}_{(p-1)}\]
	where we view the latter quotient as the essential image of the functor $L_{V_{\leq p-1}}(\mathbb{I}) \ast -$ restricted to~$\mathcal{K}_{(p)}$ via Lemma \ref{lmasubcaloc}.
	\label{quotientisgamma}
\end{lma}
\begin{proof}
	If $a$ is an object of $\Gamma_p \mathcal{K}$, we have $\supp(a) \subset V_p \subset V_{\leq p}$, so we certainly have $a \in \mathcal{K}_{(p)}$. Consider the localization triangle
	\[\Gamma_{V_{\leq p-1}} a \to a \to  L_{V_{\leq p-1}} a \to \Sigma \Gamma_{V_{\leq p-1}} a\]
	where $\supp(\Gamma_{V_{\leq p-1}} a) = V_{\leq p-1} \cap \supp(a) = \emptyset$ by Proposition \ref{propsuppprops}. Thus, we have $\Gamma_{V_{\leq p-1}} a \cong 0$ and we obtain an isomorphism $a \cong L_{V_{\leq p-1}} a$, which proves that $a$ is in the essential image of $L_{V_{\leq p-1}}(\mathbb{I}) \ast -$ restricted to~$\mathcal{K}_{(p)}$.
	
	On the other hand, if $a$ is an object of the essential image of $L_{V_{\leq p-1}}(\mathbb{I}) \ast -$ restricted to $\mathcal{K}_{(p)}$, there exists an object $a'$ of $\mathcal{K}_{(p)}$ such that $a \cong L_{V_{\leq p-1}} a'$. By Lemma \ref{subcatequality}, we know that $\supp(a') \subset V_{\leq p}$. But then by Proposition \ref{propsuppprops},
	\[\supp(a) = \supp(a') \cap (\Spc(\mathcal{T}^c) \setminus V_{\leq p-1}) \subset  V_{\leq p} \cap (\Spc(\mathcal{T}^c) \setminus V_{\leq p-1}) = V_p\]
	which proves that $a \in \Gamma_p \mathcal{K}$.
\end{proof}

We finish the section with a useful result about the subcategories $\tau_{\mathcal{K}}(V)$ for a specialization closed subset $V \subset \Spc(\mathcal{T}^c)$.
\begin{prop}
	Suppose that the action of $\mathcal{T}$ on $\mathcal{K}$ satisfies the local-to-global principle and let $V \subset \Spc(\mathcal{T}^c)$ be specialization-closed. Then $\tau_{\mathcal{K}}(V)$ is compactly generated and there is an equality of subcategories 
	\[(\tau_{\mathcal{K}}(V))^c = (\mathcal{K}^c)_V := \lbrace a \in \mathcal{K}^c: \supp(a) \subset V \rbrace~.\] 
	In particular $\mathcal{K}_{(p)} = \tau_{\mathcal{K}}(V_{\leq p})$ is compactly generated for all $p$ and $(\mathcal{K}_{(p)})^c = (\mathcal{K}^c)_{(p)}$.
	\label{propcompactsubcatsame}
\end{prop}
\begin{proof}
	As the set $V$ is specialization-closed, it follows from \cite{stevenson}*{Corollary 4.11} that $\Gamma_{V}\mathcal{K}$ is compactly generated. But $\Gamma_{V}\mathcal{K}$ is equal to $\tau_{\mathcal{K}}(V)$ by Lemma \ref{subcatequality}, and so it is compactly generated.
	
	The subcategory $\tau_{\mathcal{K}}(V)$ is precisely the kernel of the coproduct-preserving localization functor~$L_{V}(\mathbb{I}) \ast -$: indeed, if $a$ is an object of $\tau_{\mathcal{K}}(V)$, then by Proposition \ref{propsuppprops}
	\[\supp(L_{V} a) = (\Spc(\mathcal{T}) \setminus V) \cap \supp(a) = \emptyset~,\]
	which implies that $L_{V} a =0$ as we have assumed the local-to-global principle, and hence $a \in \ker(L_{V}) (\mathbb{I}) \ast -$. On the other hand, if we assume that $a \in \ker(L_{V}(\mathbb{I}) \ast -)$, then from the localization triangle
	\[\Gamma_{V} a \to a \to L_{V} a \to \Sigma \Gamma_{V} a \]
	we obtain that $a \cong \Gamma_{V} a$. But $\supp(\Gamma_{V} a) = \supp(a) \cap V \subset V$ by Proposition \ref{propsuppprops}, which implies that $a \in \tau_{\mathcal{K}}(V)$.
	
	By \cite{krauseloc}*{Proposition 5.5.1} the right adjoint of the inclusion 
	\[\iota: \tau_{\mathcal{K}}(V)= \ker(L_{V} (\mathbb{I}) \ast - ) \longrightarrow \mathcal{T}\] 
	preserves coproducts and by \cite{krauseloc}*{Lemma 5.4.1} it follows that $\iota$ preserves compactness. Therefore, $(\tau_{\mathcal{K}}(V))^c \subset (\mathcal{K}^c)_V $. The converse inclusion is an immediate consequence of the definition of compactness.
\end{proof}
In the light of the equality of subcategories of Proposition \ref{propcompactsubcatsame}, we will simply use the notation $\mathcal{K}_{(p)}^c$ for $(\mathcal{K}_{(p)})^c = (\mathcal{K}^c)_{(p)}$ if the local-to-global principle holds. Before we continue to the main definitions of the article, let us give a notational overview for the reader's convenience.

\begin{center}
	\begin{tabularx}{\textwidth}{|c|X|l|}
		\hline
		\multicolumn{3}{|c|}{\textbf{Table of notations}}\\
		\multicolumn{3}{|c|}{$x \in \Spc(\mathcal{T}^c),  a \in \mathcal{K}, p \in \mathbb{Z}$}\\
		\hline
		$\Gamma_V,L_V$ for $V \subset \Spc(\mathcal{T}^c)$ & Acyclization and localization functor associated to a specialization-closed set $V \subset \Spc(\mathcal{T}^c)$ & p.\ \pageref{locacycdefpage} \\
		\hline
		$Y_x$ & $ \left\lbrace P \in \Spc(\mathcal{T}^c) : x \notin \overline{\lbrace P \rbrace} \right\rbrace$ & p.\ \pageref{Yxdefpage}\\
		\hline
		$\Gamma_x a$ & $(\Gamma_{\overline{\lbrace x \rbrace}} L_{Y_x}(\mathbb{I})) \ast a$ & p.\ \pageref{Gammaxdefpage}\\
		\hline
		$V_{\leq p}$ & $\lbrace x \in \Spc(\mathcal{T}^c) | \dim(x) \leq p \rbrace$ & p.\ \pageref{Vleqpdefpage}\\
		\hline
		$V_p$ & $\lbrace x \in \Spc(\mathcal{T}^c) | \dim(x) = p \rbrace$ & p.\ \pageref{Vleqpdefpage}\\
		\hline
		$\Gamma_p a$ & $(\Gamma_{V_{\leq p}} L_{V_{\leq p-1}}(\mathbb{I})) \ast a$ & p.\ \pageref{Gammapdefpage}\\
		\hline
		$\mathcal{K}_{(p)}$ & $\tau_{\mathcal{K}}(V_{\leq p}) = \lbrace a \in \mathcal{K}: \supp(a) \subset V_{\leq p} \rbrace$ & p.\ \pageref{Kpdefpage}\\
		\hline
	\end{tabularx}
\end{center}

\section{Relative tensor triangular Chow groups}
\label{sectionrelchow}
In addition to the hypotheses from Convention \ref{convbigcond}, we will assume that the local-to-global principle holds for the action of $\mathcal{T}$ on~$\mathcal{K}$ for the rest of the section. For clarity's sake, we will still explicitly mention this hypothesis in the formulation of the results. 

The category $\mathcal{K}_{(p)}$ is compactly generated for all $p \in \mathbb{Z}$ by Proposition \ref{propcompactsubcatsame}. We therefore have that $(\mathcal{K}_{(p)}/\mathcal{K}_{(p-1)})^c$ is the thick closure of $\mathcal{K}_{(p)}^c/\mathcal{K}_{(p-1)}^c$ in $\mathcal{K}_{(p)}/\mathcal{K}_{(p-1)}$ (see \cite{krauseloc}*{Theorem 5.6.1}). Thus, we get an injection
\[j: \Kzero\left(\mathcal{K}_{(p)}^c/\mathcal{K}_{(p-1)}^c \right) \hookrightarrow \Kzero \left( (\mathcal{K}_{(p)}/\mathcal{K}_{(p-1)})^c \right)~.\]
Furthermore, the quotient functor $\mathcal{K}_{(p)}^c \to \mathcal{K}_{(p)}^c/\mathcal{K}_{(p-1)}^c$ and the embedding $\mathcal{K}_{(p)}^c \to \mathcal{K}_{(p+1)}^c$ induce maps
\[q: \Kzero\left(\mathcal{K}_{(p)}^c \right) \rightarrow \Kzero\left(\mathcal{K}_{(p)}^c/\mathcal{K}_{(p-1)}^c\right)\]
and
\[i:\Kzero\left(\mathcal{K}_{(p)}^c \right) \rightarrow \Kzero\left(\mathcal{K}_{(p+1)}^c \right)~.\]

\begin{dfn}\index{Relative tensor triangular Chow groups}
	We define the \emph{$p$-dimensional tensor triangular cycle groups of~$\mathcal{K}$, relative to the action of $\mathcal{T}$} and the \emph{$p$-dimensional tensor triangular Chow groups of~$\mathcal{K}$, relative to the action of $\mathcal{T}$} as follows:
	\[\Cyc^{\Delta}_p(\mathcal{T},\mathcal{K}) := \Kzero\left((\mathcal{K}_{(p)}/\mathcal{K}_{(p-1)})^c\right)\]
	and
	\[\Chow^{\Delta}_p(\mathcal{T},\mathcal{K}) := \Cyc^{\Delta}_{(p)}(\mathcal{T},\mathcal{K})/ j \circ q(\ker(i)).\]
	\label{dfnrelcycchowgroups}
\end{dfn}

\begin{rem}
	Since we assumed that the local-to-global principle is satisfied, we can view an element of $\Cyc^{\Delta}_p(\mathcal{T},\mathcal{K})$ as a finite formal sum of $p$-dimensional points $x_i$ of $\Spc(\mathcal{T}^c)$, with coefficients in $\Kzero \left(\left(\Gamma_x \mathcal{K}\right)^c\right)$ for $x \in V_p$, by Proposition \ref{stevensondecomposition} and Lemma \ref{lmacompprod}.
\end{rem}

\begin{rem}
	The category $\mathcal{K}_{(p)}/\mathcal{K}_{(p-1)}$ has arbitrary coproducts and is therefore idempotent complete (cf. \cite{neemantc}*{Proposition 1.6.8}). Since $(\mathcal{K}_{(p)}/\mathcal{K}_{(p-1)})^c$ is the thick closure of $\mathcal{K}_{(p)}^c/\mathcal{K}_{(p-1)}^c$ in $\mathcal{K}_{(p)}/\mathcal{K}_{(p-1)}$, we obtain that $(\mathcal{K}_{(p)}/\mathcal{K}_{(p-1)})^c$ is equivalent to the idempotent completion~$(\mathcal{K}_{(p)}^c/\mathcal{K}_{(p-1)}^c)^{\natural}$, cf.\ Definition~\ref{defcycchow}.
	\label{remcompactidcomp}
\end{rem}

Next, we compare $\Cyc^{\Delta}_p(\mathcal{T},\mathcal{T})$ and $\Chow^{\Delta}_p(\mathcal{T},\mathcal{T})$ to the tensor triangular cycle and Chow groups $\Cyc^{\Delta}_p(\mathcal{T}^c), \Chow^{\Delta}_p(\mathcal{T}^c)$ from Definition~\ref{defcycchow}.

\begin{prop}
	Consider the action of $\mathcal{T}$ on itself via the tensor product $\otimes$ and assume that the local-to-global principle holds for this action. Then we have isomorphisms
	\[\Cyc^{\Delta}_p(\mathcal{T},\mathcal{T}) \cong \Cyc^{\Delta}_p(\mathcal{T}^c) \quad \text{and} \quad \Chow^{\Delta}_p(\mathcal{T},\mathcal{T}) \cong \Chow^{\Delta}_p(\mathcal{T}^c)~.\]
	\label{relativecompactcomparison}
\end{prop}
\begin{proof}
	By definition we have $\Cyc^{\Delta}_p(\mathcal{T},\mathcal{T}) = \Kzero\left((\mathcal{T}_{(p)}/\mathcal{T}_{(p-1)})^c\right)$ and Remark \ref{remcompactidcomp} shows that the category $(\mathcal{T}_{(p)}/\mathcal{T}_{(p-1)})^c$ is equivalent to~$\left( \mathcal{T}^c_{(p)}/\mathcal{T}^c_{(p-1)}\right)^{\natural}$.
	We conclude that
	\[\Cyc^{\Delta}_p(\mathcal{T},\mathcal{T}) = \Kzero\left((\mathcal{T}_{(p)}/\mathcal{T}_{(p-1)})^c\right) \cong \Kzero\left(\left( \mathcal{T}^c_{(p)}/\mathcal{T}^c_{(p-1)}\right)^{\natural}\right)\]
	which is equal to $\Cyc^{\Delta}_p(\mathcal{T}^c)$ by Definition \ref{defcycchow}. The notions of rational equivalence agree as well as the maps $j,q,i$ from Definition \ref{dfnrelcycchowgroups} are equal to the corresponding maps from Definition \ref{defcycchow}.
\end{proof}

Let us work out a brief example computation which will be generalized in Section \ref{sectionkinj}. Let $X$ be a noetherian separated scheme and let $\mathrm{D}(X) := \mathrm{D}(\mathrm{Qcoh}(X))$ be the full derived category of complexes of quasi-coherent $\mathcal{O}_X$-modules. The category $\mathrm{D}(X)$ is a compactly-rigidly generated tensor triangulated category with arbitrary coproducts (see \cite{balfavibig}*{Example 1.2}), and we have $\mathrm{D}(X)^c = \Dperf(X)$ (cf.\ \cite{bondalvdbergh}*{Theorem 3.1.1}).
\begin{lma}
	Any action of $\mathrm{D}(X)$ on a compactly generated triangulated category $\mathcal{K}$ (in the sense of Section~\ref{sectionprelim}) satisfies the local-to-global principle.
	\label{lmadxlocglob}
\end{lma}
\begin{proof}
	By the criterion of \cite{stevenson}*{Proposition 6.8} it suffices to show that $\mathrm{D}(X)$ arises as the homotopy category of a monoidal model category, which is the main result of \cite{gillespiemodel}.
\end{proof}

\begin{cor}
	We have isomorphisms
	\[\Cyc^{\Delta}_p(\mathrm{D}(X),\mathrm{D}(X)) \cong \Cyc^{\Delta}_p(\Dperf(X)) \]
	and
	\[\Chow^{\Delta}_p(\mathrm{D}(X),\mathrm{D}(X)) \cong \Chow^{\Delta}_p(\Dperf(X))\]
	for all $p \in \mathbb{Z}$. In particular, if $X$ is non-singular, of finite type over a field and we equip $\mathrm{D}(X)^c$ with the opposite of the Krull codimension as a dimension function, we have
	\[\Cyc^{\Delta}_p(\mathrm{D}(X),\mathrm{D}(X)) \cong \Cyc^{-p}(X) \quad \text{and} \quad \Chow^{\Delta}_p(\mathrm{D}(X),\mathrm{D}(X)) \cong \Chow^{-p}(X).\]
	~
	\label{correlagr}
\end{cor}
\begin{proof}
	This is an immediate consequence of Proposition \ref{relativecompactcomparison} and the agreement theorem of \cite{kleinchow}. The local-to-global prinicple holds for the action of $\mathrm{D}(X)$ on itself by Lemma \ref{lmadxlocglob}.
\end{proof}

\section{The homotopy category of quasi-coherent injective sheaves and singular varieties}
\label{sectionkinj}
In this section, we illustrate the definitions given in Section \ref{sectionrelchow} by discussing the category $\mathrm{K}(\mathrm{Inj}(X))$, the homotopy category of complexes of injective quasi-coherent sheaves on a noetherian scheme $X$. The category $\mathrm{K}(\mathrm{Inj}(X))$ is a compactly generated triangulated category with 
\[\mathrm{K}(\mathrm{Inj}(X))^c \cong \Db(\mathrm{Coh}(X))~,\] 
the bounded derived category of coherent sheaves on $X$ (see \cite{krause-stablederived}). The equivalence can be given explicitly by taking a bounded complex of coherent sheaves to a quasi-isomorphic complex of injective quasi-coherent sheaves.

Let us also recall the definition of the \emph{mock homotopy category of projectives} $\mathrm{K_m}(\mathrm{Proj}(X))$ from \cite{murfet-thesis}. It is defined as the Verdier quotient $\mathrm{K}(\mathrm{Flat}(X))/\mathbf{E}(X)$, where $\mathrm{Flat}(X)$ denotes the category of flat quasi-coherent sheaves on $X$ and $\mathbf{E}(X)$ is the localizing subcategory consisting of those complexes $\mathcal{E}^{\bullet}$ such that $\mathcal{E}^{\bullet} \otimes \mathcal{F}$ is acyclic for every quasi-coherent sheaf $\mathcal{F}$. The category $\mathrm{K_m}(\mathrm{Proj}(X))$ is a compactly generated \emph{tensor} triangulated category.

\begin{prop}
	The category $\mathrm{K}(\mathrm{Inj}(X))$ admits an action of $\mathrm{D}(X)$, given by first taking $\mathrm{K}$-flat resolutions and then taking the tensor product of chain complexes. The local-to-global principle holds for this action.
	\label{propdxactsonkinjx}
\end{prop}
\begin{proof}
	The category $\mathrm{K}(\mathrm{Flat}(X))$ is a tensor-triangulated category with tensor product the usual tensor product of chain complexes. It has an action on $\mathrm{K}(\mathrm{Inj}(X))$ which is also given by said tensor product  since, on a noetherian scheme, the tensor product of a complex of flat quasi-coherent sheaves with a complex of injective quasi-coherent sheaves is of the latter type again (see \cite{murfet-thesis}*{Lemma 8.2 and preceding remark}). The axioms for an action of Section \ref{sectionprelim} are satisfied in this case, since the tensor product of chain complexes makes the homotopy category of all quasi-coherent sheaves of modules on $X$ a symmetric monoidal category which contains both $\mathrm{K}(\mathrm{Flat}(X))$ and $\mathrm{K}(\mathrm{Inj}(X))$. This action factors through $\mathrm{K_m}(\mathrm{Proj}(X))$, since by \cite{murfet-thesis}*{Lemma 8.2} the complex $F \otimes D$ is contractible when $F \in \mathbf{E}(X)$ and $D$ is a complex of injective quasi-coherent sheaves. The category $\mathrm{D}(X)$ admits a monoidal embedding into $\mathrm{K}_m(\mathrm{Proj}(X))$ by taking $\mathrm{K}$-flat resolutions (see \cite{murfet-thesis}*{Section 5}) and we define the action of $\mathrm{D}(X)$ on $\mathrm{K}(\mathrm{Inj}(X))$ as the restriction of the action of $\mathrm{K_m}(\mathrm{Proj}(X))$ on $\mathrm{K}(\mathrm{Inj}(X))$. The local-to-global principle holds for this action by Lemma \ref{lmadxlocglob}.
\end{proof}
The above implies that we have an associated support theory for the action of $\mathrm{D}(X)$ on $\mathrm{K}(\mathrm{Inj}(X))$ which assigns to every object of $\mathrm{K}$ a support in $\Spc(\mathrm{D}(X)^c) = X$.

Let us show next that the $\mathrm{K}(\mathrm{Inj}(X))$ behaves well under smashing localizations. Recall from \cite{stevenson}*{Section 8} that for an open subscheme $U \subset \Spc(\mathrm{D}(X)) = X$ with closed complement $Z$, we have an associated smashing localization $\mathrm{K}(\mathrm{Inj}(X))_U$ given as the essential image of the functor~$L_Z(\mathcal{O}_X) \ast -$. If $i: U \hookrightarrow X$ denotes the inclusion of $U$ into $X$, then $L_Z: \mathrm{D}(X) \to \mathrm{D}(X)$ is given explicitly as $\mathbf{R}i_* \circ i^*(-)$.  Denote by $\mathrm{K}(\mathrm{Inj}(X))_{Z} $ the full triangulated subcategory of those objects $C^{\bullet}$ such that
\[\bigcup_{n} \supp (C^n) \subset Z\]
where $C^{\bullet}$ is taken to be homotopically minimal (see \cite{krause-stablederived}).
\begin{lma}[see {\cite[Proposition 6.9]{krause-stablederived}}]
	Let $i: U \to X$ denote the inclusion of the open subscheme $U$. The functor
	\[i^*: \mathrm{K}(\mathrm{Inj}(X)) \to \mathrm{K}(\mathrm{Inj}(U))\]
	induces an equivalence
	\[\mathrm{K}(\mathrm{Inj}(X))/\mathrm{K}(\mathrm{Inj}(X))_{Z} \cong \mathrm{K}(\mathrm{Inj}(U))~.\]
	\label{lmarestoopen}
\end{lma}

\begin{rem}
The associated localization functor on $\mathrm{K}(\mathrm{Inj}(X))$ is given by $i_* \circ i^*(-)$ (see \cite{krause-stablederived}). It realizes the Verdier quotient functor inside $\mathrm{K}(\mathrm{Inj}(X))$.
\label{remreslocfunc}
\end{rem}

\begin{lma}
	Let $f: S \to T$ be an affine morphism. Then the functor $f_* \circ f^*(-)$ is isomorphic to $f_*(\mathcal{O}_S) \otimes_{\mathcal{O}_T} -$. In particular, if $f$ is also flat, the sheaf $f_{*}(\mathcal{O}_S)$ is flat.
	\label{lmaaffpullpush}
\end{lma}
\begin{proof}
	The category $\mathrm{Qcoh}(S)$ is equivalent to the category of $f_*(\mathcal{O}_S)$-modules on $T$ which are quasi-coherent as $\mathcal{O}_T$-modules. Under this identification, the functors $f^*$ and $f_*$ are given as $f_*(\mathcal{O}_S) \otimes_{\mathcal{O}_T} -$ and the forgetful functor respectively. It follows that for every quasi-coherent sheaf $\mathcal{E}$ on $T$ the $\mathcal{O}_T$-module $f_* \circ f^*(\mathcal{E})$ is given by $f_* \mathcal{O}_S \otimes_{\mathcal{O}_T} \mathcal{E}$.
	
	If we assume $f$ affine and flat, the flatness of $f_{*}(\mathcal{O}_S)$ follows as both $f^*$ and $f_*$ are exact in this case, and therefore so is their composition. Thus, $f_*(\mathcal{O}_S) \otimes_{\mathcal{O}_T} -$ is exact, which by definition means that  $f_*(\mathcal{O}_S)$ is flat.
\end{proof}

\begin{cor}
	Assume that $X$ is a noetherian \emph{separated} scheme and let $i:U \to X$ denote the inclusion of an \emph{affine} open subscheme with closed complement $Z$. Then $\mathrm{K}(\mathrm{Inj}(X))_U$ is equivalent to $\mathrm{K}(\mathrm{Inj}(U))$. 
	\label{corinjres}
\end{cor}
\begin{proof}
	By definition, the category $\mathrm{K}(\mathrm{Inj}(X))_U$ is given as the essential image of the functor~$L_Z(\mathcal{O}_X) \ast -$ with~$L_Z: \mathrm{D}(X) \to \mathrm{D}(X)$. But $L_Z$ can be described as the functor $\mathbf{R}i_* \circ i^*$, and thus $\mathrm{K}(\mathrm{Inj}(X))_U$ is the essential image of the functor $\mathbf{R}i_* \circ i^*(\mathcal{O}_X) \ast - = i_*i^*(\mathcal{O}_X) \otimes -$, where we don't need to derive $i_*$ because $X$ is separated (which implies that $i$ is affine by \cite{stacks-project}*{Tag 01SG} and therefore $i_*$ exact) and $\ast = \otimes$ in this case, since $i_*i^*(\mathcal{O}_X) = i_*(\mathcal{O}_U)$ is a flat sheaf on $X$ by Lemma~\ref{lmaaffpullpush}. The projection formula yields the isomorphism of functors $i_*i^*(\mathcal{O}_X) \otimes - \cong i_*i^*(-)$ and by Lemma \ref{lmarestoopen} and Remark \ref{remreslocfunc} it follows that 
	\[\im(L_Z(\mathcal{O}_X) \ast -) = \im(i_*i^*(-)) = \mathrm{K}(\mathrm{Inj}(U))~.\]
\end{proof}

Let us check now that the notion of support obtained from the action of $\mathrm{D}(X)$ on $\mathrm{K}(\mathrm{Inj}(X))$ is a familiar one.
\begin{prop}
	Assume that $X$ is a noetherian separated scheme. Let $\mathcal{E}^{\bullet}$ be a cohomologically bounded complex of sheaves with coherent cohomology in $\mathrm{K}(\mathrm{Inj}(X))$. Then $\supp(\mathcal{E}^{\bullet}) = \mathrm{supph}(\mathcal{E}^{\bullet})$, where $\mathrm{supph}(\mathcal{E}^{\bullet})$ denotes the cohomological support of the chain complex $\mathcal{E}^{\bullet}$.
	\label{proposition-suppagrees}
\end{prop}
\begin{proof}
	First note that for any action $\ast$ of a tensor triangulated category $\mathcal{T}$ on a triangulated category $\mathcal{K}$ and quasi-compact open $U \subset \Spc(\mathcal{T}^c)$ with closed complement $Z$, we have an induced action $\ast_U$ of $\mathcal{T}(U) := L_Z(\mathcal{T})$ on $\mathcal{K}(U) := L_Z(\mathbb{I}) \ast \mathcal{K}$ by \cite[Proposition 8.5]{stevenson}, which makes the following diagram commute (and is, in fact, defined by this property): 
	\begin{equation}
	\xymatrix@C=7em{
		\mathcal{T} \times \mathcal{K} \ar[r]^-{L_Z(-) \times L_Z(\mathbb{I}) \ast -} \ar[d]^{\ast}& \mathcal{T}(U) \times \mathcal{K}(U) \ar[d]^{\ast_U} \\
		\mathcal{K} \ar[r]^-{L_Z(\mathbb{I}) \ast -} & \mathcal{K}(U)
	}
	\label{eqninducedactiondiag}
	\end{equation}
	
	Identifying $\Spc(\mathcal{T}(U)^c)$ with $U$ and taking a cover $(U_i)_{i \in I}$ of $\Spc(\mathcal{T}^c)$ consisting of quasi-compact opens one then obtains the formula
	\begin{equation}
	\supp_{\mathcal{T}}(A) = \bigcup_{i} \supp_{\mathcal{T}(U_i)} L_{Z_i}(\mathbb{I}) \ast A
	\label{eqnsupplocal}
	\end{equation}
	where $Z_i$ denotes the complement of $U_i$ in $\Spc(\mathcal{T}^c)$ (see \cite[Remark 8.9]{stevenson}).
	
	In our specific situation, this is used as follows: we take an affine open cover of  $(U_i)_{i \in I}$ of $\Spc(\mathrm{D}(X)^c) = X$ and obtain for every open $U_i \subset X$ with complement $Z_i$ an action of the category $\mathrm{D}(X)(U_i) = L_{Z_i}(\mathrm{D}(X)) \cong \mathrm{D}(U_i)$ on $\mathrm{K}(\mathrm{Inj}(U_i))= L_{Z}(\mathbb{I}) \ast \mathrm{K}(\mathrm{Inj}(U_i)) \cong \mathrm{K}(\mathrm{Inj}(U_i))$. From the commutativity of diagram (\ref{eqninducedactiondiag}) we see that this is just the usual action obtained by taking $\mathrm{K}$-flat resolutions and taking tensor products. Using the identities (\ref{eqnsupplocal}) and
	\[\mathrm{supph}_X(F^{\bullet}) = \bigcup_{i} \mathrm{supph}_{U_i}(F^{\bullet}|_{U_i})\]
	it suffices to show that for  each affine patch $U_i = \mathrm{Spec}(A_i)$ and every homologically bounded complex $E^{\bullet}$ of injective $A_i$-modules with finitely generated total cohomology, $\supp_{\mathrm{D}(A_i)}(E^{\bullet}) = \mathrm{supph}(E^{\bullet})$. By definition of the action of $\mathrm{D}(U_i)$ on $\mathrm{K}(\mathrm{Inj}(U_i))$, the set $\supp_{\mathrm{D}(A_i)}(E^{\bullet})$ is given as
	\[\lbrace \mathfrak{p} \in \mathrm{Spec}(A_i):  \Gamma_{\mathfrak{p}} \otimes^L E^{\bullet} \neq 0 \rbrace~.\]
	Using \cite[Proposition 3.9]{stevenson-singularity}, we see that $\Gamma_{\mathfrak{p}}$ is given as $\mathrm{K}_{\infty}(\mathfrak{p})\otimes (A_i)_{\mathfrak{p}}$, where $\mathrm{K}_{\infty}(\mathfrak{p})$ is the stable Kozsul complex at $\mathfrak{p}$. Hence, if $\mathbf{R}\Gamma_{\mathfrak{p}}$ denotes the local cohomology functor at $\mathfrak{p}$, we have
	\[\Gamma_{\mathfrak{p}} \otimes^L E^{\bullet} \cong \mathbf{R}\Gamma_{\mathfrak{p}}(E^{\bullet}_{\mathfrak{p}})\]
	and it follows from \cite{bik-localcohomsupp}*{Section 9} that 
	\[\lbrace \mathfrak{p} \in \mathrm{Spec}(A_i):  \mathbf{R}\Gamma_{\mathfrak{p}}(E^{\bullet}_{\mathfrak{p}}) \neq 0 \rbrace = \mathrm{supph}(E^{\bullet})\]
	 when $\mathrm{H}^*(E^{\bullet})$ is bounded and finitely generated.

\end{proof}

\begin{cor}
	Let $X$ be a noetherian separated scheme and $p \in \mathbb{Z}$. Then the category $(\mathrm{K})^{c}_{(p)}$ is equivalent to the full subcategory $\Db(X)_{(p)} \subset \Db(X)$ consisting of those complexes $E^{\bullet}$ such that $\dim(\mathrm{supph}(E^{\bullet})) \leq p$.
\end{cor}
\begin{proof}
	The subcategory of compact objects of $\mathrm{K}(\mathrm{Inj}(X))$ can be identified with the full subcategory of those complexes that come from injective resolutions of objects in $\Db(X)$, and thus they are cohomologically bounded and have coherent cohomology. By Proposition \ref{proposition-suppagrees}, it follows that for these complexes, $\supp$ coincides with $\mathrm{supph}$, and so $\mathrm{K}(\mathrm{Inj}(X))^{c}_{(p)}$ consists of those complexes in $\mathrm{K}(\mathrm{Inj}(X))^{c}$ whose cohomology is supported in dimension~$\leq p$. These exactly correspond to the complexes of $\Db(X)$ whose cohomology is supported in dimension~$\leq p$.
\end{proof}

From this, we can compute $\Cyc^{\Delta}_{p}(\mathrm{D}(X),\mathrm{K}(\mathrm{Inj}(X)))$ and~$\Chow^{\Delta}_{p}(\mathrm{D}(X),\mathrm{K}(\mathrm{Inj}(X)))$.
\begin{thm}
	Let $X$ be a separated scheme of finite type over a field. Endow $\Spc(\mathrm{D}(X)^c) = X$ with the opposite of the Krull codimension as a dimension function. Then for all $p \leq 0$, we have isomorphisms
	\begin{align*}
	\Cyc^{\Delta}_{p}(\mathrm{D}(X),\mathrm{K}(\mathrm{Inj}(X))) &\cong \Cyc_{-p}(X) \\
	\Chow^{\Delta}_{p}(\mathrm{D}(X),\mathrm{K}(\mathrm{Inj}(X))) &\cong \Chow_{-p}(X)~.
	\end{align*}
	\label{thmsingagreement}
\end{thm}
\begin{proof}
	In \cite{kleinchow}*{Section 4} it is proved that $\Kzero(\Db(X)_{(p)}/\Db(X)_{(p-1)}(X)) \cong \Cyc_{-p}(X)$ and $j \circ q (\ker(i))$ (see Definition \ref{dfnrelcycchowgroups}) equals the subgroup of cycles rationally equivalent to zero under this isomorphism.
\end{proof}

\section{Restriction to open subsets} 
Let $U \subset \Spc(\mathcal{T}^c)$ be an open subset with complement $Z$. If we denote by $\mathcal{T}_Z$ the smashing ideal in $\mathcal{T}$ generated by the subcategory $(\mathcal{T}^c)_Z \subset \mathcal{T}^c$ of all objects with support contained in $Z$, then the quotient category $\mathcal{T}_U := \mathcal{T}/\mathcal{T}_Z$ is a tensor triangulated category satisfying all the assumptions made at the beginning of this chapter, whose spectrum $\Spc(\mathcal{T}_U^c)$ can be identified with $U$. Furthermore, if $\mathcal{T}$ acts on $\mathcal{K}$, we obtain an induced action on the category $\mathcal{K}_U := L_Z(\mathbb{I}) \ast \mathcal{K}$ (see \cite{stevenson}*{Proposition 8.5}). We will show that the localization functor induces surjective maps 
\[\Cyc^{\Delta}_p(\mathcal{T},\mathcal{K}) \to \Cyc^{\Delta}_p(\mathcal{T}_U,\mathcal{K}_U)\]
and
\[\Chow^{\Delta}_p(\mathcal{T},\mathcal{K}) \to \Chow^{\Delta}_p(\mathcal{T}_U,\mathcal{K}_U)\]
for all $p$. The kernels of these maps can be described with the help of the relative cycle and Chow groups that we introduced in the previous section.

\begin{lma}
	Let $\mathcal{T}$ be a compactly-rigidly generated  tensor triangulated category (see Definition \ref{dfncompriggen}) such that $\Spc(\mathcal{T}^c)$ is a noetherian topological space. Let $U \subset \Spc(\mathcal{T}^c)$ be an open subset with complement $Z$. Then $\mathcal{T}_U$ is a compactly-rigidly generated tensor triangulated category with $\Spc(\mathcal{T}_U^c) \cong U$, which is a noetherian topological space as well. Furthermore, the action $\ast$ induces an action $\ast_U$ of $\mathcal{T}_U$ on $\mathcal{K}_U := L_Z(\mathbb{I}) \ast \mathcal{K}$, and if the local-to-global principle holds for the action $\ast$, it also holds for the action $\ast_U$. 
	\label{lmaresaction}
\end{lma}
\begin{proof}
	The category $\mathcal{T}_U$ is compactly-rigidly generated and has an action on $\mathcal{K}_U$ by \cite{stevenson}*{Section 8}. Furthermore $\Spc(\mathcal{T}_U^c) \cong U$, since $\mathcal{T}_U^c = (\mathcal{T}_c/(\mathcal{T}_c)_Z)^{\natural}$ by \cite{krauseloc}*{Theorem 5.6.1 }. The space $U$ is noetherian as it is a subspace of a noetherian topological space by \cite{stacks-project}*{Tag 0052}. For the last statement, note that any localizing $\mathcal{T}_U$-submodule $\mathcal{M} \subset \mathcal{K}_U$ is a localizing $\mathcal{T}$-submodule of $\mathcal{K}$ as well. Indeed, for objects $a \in \mathcal{M}, s \in \mathcal{T}$, we have $s \ast a \cong s \ast (L_Z s) \cong (s \otimes L_Z(\mathbb{I})) \ast A \in \mathcal{M}$ since $s \otimes L_Z(\mathbb{I}) \in \mathcal{T}_U$. Thus it suffices to show that for $a \in \mathcal{K}_U$, we have that $\langle a \rangle_{\ast} = \langle \Gamma_x(\mathbb{I}) \ast a| x \in \Spc(\mathcal{T}_U^c) \rangle_{\ast}$, where we interpret $\langle - \rangle_{\ast}$ as the smallest localizing $\mathcal{T}$-submodule of $\mathcal{K}$ containing~$-$. Then we have
	\[ \langle a \rangle_{\ast} = \langle \Gamma_x a | x \in \Spc(\mathcal{T}^c) \rangle_{\ast}\]
	by the local-to-global principle for the action of $\mathcal{T}$ on $\mathcal{K}$. But $\Gamma_x(\mathbb{I}) \otimes L_Z(\mathbb{I}) = 0$ if $x \notin U$ and  $\Gamma_x(\mathbb{I})  \otimes L_Z(\mathbb{I}) \cong \Gamma_x(\mathbb{I}) $ if $x \in U$ (see \cite{stevenson}*{Proposition 8.3}), so
	\begin{align*}
	\langle a \rangle_{\ast} = \langle \Gamma_x a | x \in \Spc(\mathcal{T}^c) \rangle_{\ast} & = \langle (\Gamma_x(\mathbb{I}) \otimes L_Z(\mathbb{I})) \ast a | x \in \Spc(\mathcal{T}^c) \rangle_{\ast} \\
	& = \langle \Gamma_x a | x \in U \rangle_{\ast}\\
	& = \langle \Gamma_x a | x \in  \Spc(\mathcal{T}_U^c)\rangle_{\ast}~,
	\end{align*}
	where the last line follows from the homeomorphism $\Spc(\mathcal{T}_U^c) \cong U$. This finishes the proof.
\end{proof}

For the rest of the section, we assume that $\mathcal{T}$ is a tensor triangulated category in the sense of Convention \ref{convbigcond}, acting on $\mathcal{K}$, and that the local-to-global principle holds for this action. For any open subset $U \subset \Spc(\mathcal{T}^c)$, we will equip  $\mathcal{T}_U^c$ with the dimension function $\dim|_U$ obtained as the restriction to $U$ of the dimension function $\dim$ on $\mathcal{T}^c$. More precisely, if $\mathrm{inc}_U$ denotes the inclusion $U \hookrightarrow \Spc(\mathcal{T}^c)$, we set
\begin{align*}
\dim|_U:~ U = \Spc(\mathcal{T}_U^c) &\to \mathbb{Z} \cup \lbrace \pm \infty \rbrace\\
P &\mapsto \dim(\mathrm{inc}_U(P))~.
\end{align*}

\begin{rem}
	We warn the reader that the dimension function $\dim|_U$ may not always be the expected one. For example, if $R$ is a discrete valuation ring, $\mathcal{T} = \mathrm{D}(R)$ then $\Spc(\mathcal{T}^c) = \Spc(\Dperf(R)) = \mathrm{Spec}(R)$ (see Section \ref{sectionprelim}) and $\mathrm{Spec}(R) = \lbrace (0), \mathfrak{m} \rbrace$ , with $\mathfrak{m}$ the unique closed point and $(0)$ the generic point of the space. If we equip $\mathrm{Spec}(R)$ with the Krull dimension as a dimension function and let $U:= \lbrace (0) \rbrace$, then the space $U$ has Krull dimension $0$, but its single point satisfies~$\dim|_U((0)) = 1$! Note that this discrepancy would not have occurred if we had chosen the oppposite of the Krull codimension as a dimension function instead. This is true in general (see \cite{kleinintersection}*{Lemma 2.18}).
\end{rem}

With this convention, we obtain:
\begin{lma}
	The localization functor $L_Z(\mathbb{I}) \ast -: \mathcal{K} \to \mathcal{K}_{U}$ induces group homomorphisms
	\[l_p:\Cyc^{\Delta}_p(\mathcal{T},\mathcal{K}) \to \Cyc^{\Delta}_p(\mathcal{T}_U,\mathcal{K}_U)\]
	and
	\[\ell_p:\Chow^{\Delta}_p(\mathcal{T},\mathcal{K}) \to \Chow^{\Delta}_p(\mathcal{T}_U,\mathcal{K}_U)\]
	for all $p \in \mathbb{Z}$.
	\label{lmaresmap}
\end{lma}
\begin{proof}
	Since the localization functor $L_Z(\mathbb{I}) \ast -$ is smashing, it follows from \cite{krauseloc}*{Proposition 5.5.1 and Lemma 5.4.1} that it restricts to a functor
	\[\mathcal{K}^c \to (\mathcal{K}_{U})^c\]
	on the level of compact objects. For any object $A \in \mathcal{K}^c$ we have that $\supp_{\mathcal{T}_U}(L_Z(\mathbb{I}) \ast A) = \supp_{\mathcal{T}}(A) \cap U$ under the identification $\Spc(\mathcal{T}_U^c) = U$ and therefore $L_Z(\mathbb{I}) \ast \mathcal{K}^c_{(p)} \subset \left(\mathcal{K}_{U}\right)^c_{(p)}$ for all $p \in \mathbb{Z}$. By the universal properties of the Verdier quotient and idempotent completion we obtain an induced functor
	\[F: \left(\mathcal{K}^c_{(p)}/\mathcal{K}^c_{(p-1)} \right)^{\natural} \to \left(\left(\mathcal{K}_{U}\right)^c_{(p)}/\left(\mathcal{K}_{U}\right)^c_{(p-1)}\right)^{\natural}\]
	that fits into a commutative diagram
	\[
	\xymatrix{
		\mathcal{K}^c_{(p)} \ar[r] \ar[d] \ar[rd] & \left(\mathcal{K}^c_{(p)}/\mathcal{K}^c_{(p-1)} \right)^{\natural} \ar[rd]^F& \\
		\mathcal{K}^c_{(p+1)} \ar[rd(0.8)]&  \left(\mathcal{K}_{U}\right)^c_{(p)} \ar[r] \ar[d] & \left(\left(\mathcal{K}_{U}\right)^c_{(p)}/\left(\mathcal{K}_{U}\right)^c_{(p-1)}\right)^{\natural} \\
		& \left(\mathcal{K}_{U}\right)^c_{(p+1)} &
	}
	\]
	where the diagonal arrows are all induced by the localization $L_Z(\mathbb{I}) \ast -$, the vertical arrows are inclusions and the horizontal ones Verdier quotients composed with inclusions into the respective idempotent completions. Applying $\Kzero(-)$ to the diagram then yields the desired morphisms $l_p$ and $\ell_p$.
\end{proof}

Let us now prove that $l_p,\ell_p$ are surjective.
\begin{prop}
	The maps $l_p,\ell_p$ from Lemma \ref{lmaresmap} are surjective.
	\label{propellsurj}
\end{prop}
\begin{proof}
	It suffices to prove the statement for $l_p$ since its surjectivity implies the surjectivity of~$\ell_p$. Recall from Lemma \ref{quotientisgamma} and Proposition \ref{stevensondecomposition} that we have a decomposition
	\[\mathcal{K}_{(p)}/\mathcal{K}_{(p-1)} \cong \prod_{x \in V_p} \Gamma_x \mathcal{K}~.\]
	The subcategory $\mathcal{K}_{(p-1)}\subset \mathcal{K}_{(p)}$ is compactly generated by Proposition \ref{propcompactsubcatsame} and hence it follows from \cite{krauseloc}*{Theorem 5.6.1} and Lemma \ref{lmacompprod} that
	\[\left(\mathcal{K}_{(p)}^c/\mathcal{K}_{(p-1)}^c\right)^{\natural} \cong \left(\mathcal{K}_{(p)}/\mathcal{K}_{(p-1)}\right)^c \cong \coprod_{x \in V_p} \left(\Gamma_x \mathcal{K}\right)^c~.\] 
	Using that $\Gamma_x \mathcal{K}_U \cong \Gamma_x \mathcal{K}$ for all $x \in U$ (see \cite{stevenson}*{Proposition 8.6}), we also obtain
	\[\left(\left(\mathcal{K}_{U}\right)^c_{(p)}/\left(\mathcal{K}_{U}\right)^c_{(p-1)}\right)^{\natural} \cong \coprod_{x \in V_p \cap U} \left(\Gamma_x \mathcal{K}\right)^c~.\]
	Using this description, the quotient functor $L_Z(\mathbb{I}) \ast -$ is, up to equivalence, simply given as the natural projection 
	\[\coprod_{x \in V_p} \left(\Gamma_x \mathcal{K}\right)^c \to \coprod_{x \in V_p \cap U} \left(\Gamma_x \mathcal{K}\right)^c\]
	since $L_Z(\mathbb{I}) \ast -$ acts as a self-equivalence on $\Gamma_x \mathcal{K}$ when $x \in U$ and as $0$ otherwise, as follows from Proposition \ref{propsuppprops}. It follows that the induced map $l_p$ on Grothendieck groups is surjective.
\end{proof}

Next we try to identify the kernels of $l_p$ and $\ell_p$ with the help of the relative Chow groups.
\begin{lma}
	The category $\mathcal{K}_Z := \Gamma_Z(\mathbb{I}) \ast \mathcal{K}$ is a compactly generated, localizing $\mathcal{T}$-submodule of $\mathcal{K}$, and it coincides with the full subcategory $\tau_{\mathcal{K}}(Z)$ of objects with support contained in $Z$. Furthermore the local-to-global principle holds for the action of $\mathcal{T}$ on $\mathcal{K}_Z$.
	\label{lmasubcataction}
\end{lma}
\begin{proof}
	It is clear that the category $\mathcal{K}_Z$ is a localizing $\mathcal{T}$-submodule and it is compactly generated by \cite{stevenson}*{Corollary 4.11}. Furthermore, we have $\tau_{\mathcal{K}}(Z) = \Gamma_Z(\mathbb{I}) \ast \mathcal{K}$ by Lemma \ref{subcatequality}.
	In order to show that the local-to-global principle holds for the action of $\mathcal{T}$ on $\mathcal{K}_Z$, notice first that a localizing $\mathcal{T}$-submodule of $\mathcal{K}_Z$ is the same as a localizing $\mathcal{T}$-submodule of $\mathcal{K}$ contained in $\mathcal{T}_Z$. Hence, it suffices to show  that for any object $a \in \mathcal{K}_Z$ we have an equality of localizing $\mathcal{T}$-submodules of $\mathcal{K}$
	\[\langle a \rangle_{\ast} = \langle \Gamma_x(\mathbb{I}) \ast a | x \in \Spc(\mathcal{T}^c) \rangle_{\ast}~,\]
	which is true since we assumed the local-to-global principle for the action of $\mathcal{T}$ on $\mathcal{K}$.
\end{proof}

\begin{lma}
	For all $p \in \mathbb{Z}$ there is an exact equivalence
	\[(\mathcal{K}_Z)_{(p)}/(\mathcal{K}_Z)_{(p-1)} \cong \prod_{x \in V_p \cap Z} \Gamma_x \mathcal{K}\]
	and hence
	\[\left((\mathcal{K}_Z)_{(p)}/(\mathcal{K}_Z)_{(p-1)}\right)^c \cong \coprod_{x \in V_p \cap Z} \left(\Gamma_x \mathcal{K}\right)^c~.\]
	\label{lmasubquotientkz}
\end{lma}
\begin{proof}
	The second statement follows from the first by Lemma \ref{lmacompprod}. By Proposition \ref{stevensondecomposition}, we have
	\[(\mathcal{K}_Z)_{(p)}/(\mathcal{K}_Z)_{(p-1)} \cong \prod_{x \in V_p} \Gamma_x (\mathcal{K}_Z)\]
	and hence, in order to show the first statement, it suffices to prove that
	\[\Gamma_x (\mathcal{K}_Z) = \begin{cases} \Gamma_x \mathcal{K} & \text{if}~ x \in Z \\ 0 & \text{otherwise} ~. \end{cases}\]
	Since any object of $\mathcal{K}_Z$ is isomorphic to $\Gamma_Z a$ for some object $a \in \mathcal{K}$, we reduce the proof of this statement to showing that 
	\[\Gamma_x \Gamma_Z a \cong \begin{cases} \Gamma_x a & \text{if}~ x \in Z \\ 0 & \text{otherwise}~. \end{cases}\]
	This exact statement is proved in \cite{stevenson}*{Proof of Proposition 5.7}.
\end{proof}

\begin{lma}
	The inclusion functor
	\[\mathcal{K}_Z \hookrightarrow \mathcal{K}\]
	induces group homomorphisms
	\[i_p: \Cyc^{\Delta}_p(\mathcal{T},\mathcal{K}_Z) \to \Cyc^{\Delta}_p(\mathcal{T},\mathcal{K})\]
	and
	\[\iota_p: \Chow^{\Delta}_p(\mathcal{T},\mathcal{K}_Z) \to \Chow^{\Delta}_p(\mathcal{T}, \mathcal{K})\]
	for all $p \in \mathbb{Z}$.
\end{lma}
\begin{proof}
	Again, it follows from \cite{krauseloc}*{Proposition 5.5.1 and Lemma 5.4.1} that the inclusion functor restricts to the level of compact objects:
	\[I: \mathcal{K}_Z^c \hookrightarrow \mathcal{K}^c~.\]
	By the universal property of Verdier localization and idempotent completion, one obtains induced functors
	\[\underline{I}_{p}: \left((\mathcal{K}_Z^c)_{(p)}/(\mathcal{K}_Z^c)_{(p-1)} \right)^{\natural} \longrightarrow \left((\mathcal{K}^c)_{(p)}/(\mathcal{K}^c)_{(p-1)}\right)^{\natural}~.\]
	As we saw in Lemma \ref{lmasubcataction} the category $\mathcal{K}_Z$ is compactly generated, has set-indexed coproducts and a natural action by $\mathcal{T}$ that satisfies the local-to-global principle. Thus it makes sense to talk about the relative cycle groups $\Cyc^{\Delta}_p(\mathcal{K},\mathcal{K}_Z)$ and by the discussion at the beginning of Section \ref{sectionrelchow} we have that 
	\[\Cyc^{\Delta}_p(\mathcal{T},\mathcal{K}_Z) \cong \Kzero\left(\left((\mathcal{K}_Z^c)_{(p)}/(\mathcal{K}_Z^c)_{(p-1)} \right)^{\natural}\right)~.\]
	We see that after applying $\Kzero$, the functor $\underline{I}_{p}$ induces a map 
	\[i_p: \Cyc^{\Delta}_p(\mathcal{T},\mathcal{K}_Z) \to \Cyc^{\Delta}_p(\mathcal{T},\mathcal{K})\]
	and this map respects rational equivalence since $I$ sends 
	\[\ker\left(\Kzero\left(\mathcal{K}_Z^c)_{(p)}\right)  \to \Kzero\left((\mathcal{K}_Z^c)_{(p+1)}\right)\right)\]
	to
	\[\ker\left(\Kzero\left(\mathcal{K}^c_{(p)}\right)  \to \Kzero\left(\mathcal{K}^c_{(p+1)}\right)\right)~.\]
	Thus we also get an induced map
	\[\iota_p: \Chow^{\Delta}_p(\mathcal{T},\mathcal{K}_Z) \to \Chow^{\Delta}_p(\mathcal{T},\mathcal{K})\]
	as desired.
\end{proof}

\begin{prop}
	The map $i_p$ is injective and we have an equality of abelian groups
	\[\im(i_p) = \ker(l_p)\]
	for all $p \in \mathbb{Z}$. Hence, $\im(\iota_p) \subseteq \ker(\ell_p)$.
	\label{propkerimeq}
\end{prop}
\begin{proof}
	Using the descriptions of Lemma \ref{stevensondecomposition} and \ref{lmasubquotientkz}, the functor $\underline{I}_{p}$ is given as the canonical inclusion
	\[\left((\mathcal{K}_Z)_{(p)}/(\mathcal{K}_Z)_{(p-1)}\right)^c \cong \coprod_{x \in V_p \cap Z} \left(\Gamma_x \mathcal{K}\right)^c \hookrightarrow \coprod_{x \in V_p} \left(\Gamma_x \mathcal{K}\right)^c \cong \left(\mathcal{K}_{(p)}/\mathcal{K}_{(p-1)}\right)^c~,\]
	and hence the induced map $i_p = \Kzero(\underline{I}_{p}): \Cyc^{\Delta}_p(\mathcal{T},\mathcal{K}_Z) \to \Cyc^{\Delta}_p(\mathcal{T},\mathcal{K})$ is injective. Recall from the proof of Proposition \ref{propellsurj} that $l_p$ is induced by the projection functor
	\[\coprod_{x \in V_p} \left(\Gamma_x \mathcal{K}\right)^c \to \coprod_{x \in V_p \cap U} \left(\Gamma_x \mathcal{K}\right)^c\]
	from which we see that
	\[\ker(l_p) = \coprod_{x \in V_p \cap Z} \left(\Gamma_x \mathcal{K}\right)^c \subset \coprod_{x \in V_p} \left(\Gamma_x \mathcal{K}\right)^c~.\]
	This conincides with the image of $i_p$ by the previous discussion.
\end{proof}

The following theorem summarizes the results that we obtained so far.
\begin{thm}
	Let $\mathcal{T}$ be a compactly-rigidly generated tensor triangulated category  acting on a triangulated category $\mathcal{K}$ such that the action satisfies the local-to-global principle. Let $U \subset \Spc(\mathcal{T}^c)$ be an open subset with closed complement $Z$, denote by $\mathcal{K}_Z \subset \mathcal{K}$ the full subcategory of those objects $a$ with $\supp(a) \subset Z$ and let $\mathcal{K}_U := \mathcal{K}/\mathcal{K}_Z$. Then the quotient functor $\mathcal{K} \to \mathcal{K}_U$ and the inclusion $\mathcal{K}_Z \to \mathcal{K}$ induce exact localization sequences
	\[0 \to \Cyc^{\Delta}_p(\mathcal{T},\mathcal{K}_Z) \xrightarrow{i_p} \Cyc^{\Delta}_p(\mathcal{T}, \mathcal{K}) \xrightarrow{l_p} \Cyc^{\Delta}_p(\mathcal{T}_U, \mathcal{K}_U) \to 0\]
	and
	\[\Chow^{\Delta}_p(\mathcal{T}, \mathcal{K}) \xrightarrow{\ell_p} \Chow^{\Delta}_p(\mathcal{T}_U, \mathcal{K}_U) \to 0\]
	for all $p \in \mathbb{Z}$. Furthermore, $i_P$ induces a map $\iota_p: \Chow^{\Delta}_p(\mathcal{T}, \mathcal{K}_Z) \to \Chow^{\Delta}_p(\mathcal{T}, \mathcal{K})$ such that $\im(\iota_p) \subset \ker(\ell_p)$.
	\label{thmlocseq}
\end{thm}

\begin{rem}
	We do not know whether we also have $\im(\iota_p) = \ker(\ell_p)$ in general. However, under an additional hypothesis, we are able to prove $\ker(\ell_p) = 0$ when $p > \dim(Z)$, see Proposition \ref{propchowresiso}.
\end{rem}

\begin{rem}
	When $\mathcal{K} = \mathcal{T}$, Theorem \ref{thmlocseq} combines with Proposition \ref{relativecompactcomparison} and gives short exact sequences
	\[0 \to \Cyc^{\Delta}_p(\mathcal{T},\mathcal{T}_Z) \to  \Cyc^{\Delta}_p(\mathcal{T}^c) \to \Cyc^{\Delta}_p((\mathcal{T}_U)^c) \to 0\]
	and
	\[\Chow^{\Delta}_p(\mathcal{T}^c) \to \Chow^{\Delta}_p((\mathcal{T}_U)^c) \to 0\]
	These exact sequences should be compared to the corresponding ones for cycle and Chow groups of algebraic varieties (see \cite{fulton}*{Proposition 1.8}): if $X$ is an algebraic variety, $Z \subset X$ a closed subscheme with open complement $U$, then we get exact sequences of cycle groups
	\[0 \to \Cyc_p(Z) \to \Cyc_p(X) \to  \Cyc_p(U) \to 0\]
	and Chow groups
	\[\Chow_p(Z) \to \Chow_p(X) \to  \Chow_p(U) \to 0\]
	for all $p \in \mathbb{Z}$.
	\label{remvarexseq}
\end{rem}

The localization sequence for relative cycle groups implies that $l_p: \Cyc^{\Delta}_p(\mathcal{T}, \mathcal{K}) \to \Cyc^{\Delta}_p(\mathcal{T}_U, \mathcal{K}_U) $ is an isomorphism when $p > \dim(Z)$: in this case, $\Cyc^{\Delta}_p(\mathcal{T},\mathcal{K}_Z) = 0$ since every object of $\mathcal{K}_Z^c$ has dimension of support~$<p$. Let us finish by showing that with one extra hypothesis, we also obtain isomorphisms on the level of Chow groups in this situation.
\begin{prop}
	Assume that $p > \dim(Z)$ and that $\mathcal{K}^c/\mathcal{K}^c_Z$ is idempotent complete. Then 
	\[\ell_p: \Chow^{\Delta}_p(\mathcal{T}, \mathcal{K}) \to \Chow^{\Delta}_p(\mathcal{T}_U, \mathcal{K}_U) \]
	is an isomorphism.
	\label{propchowresiso}
\end{prop}
\begin{proof}
	Let $q \geq p$. First note that $p > \dim(Z)$ implies $\mathcal{K}_{Z} \subset \mathcal{K}_{(q)}$ and hence $(\mathcal{K}_U)_{(q)} = \mathcal{K}_{(q)}/\mathcal{K}_Z$ by the definition of the dimension function on $\mathcal{T}_U$. Since we assumed $\mathcal{K}^c/\mathcal{K}^c_Z$ idempotent complete, the same holds for $\mathcal{K}^c_{(q)}/\mathcal{K}^c_Z$ and in particular the latter category is identified with $(\mathcal{K}_U)^c_{(q)}$ by \cite{krauseloc}*{Theorem 5.6.1}. Thus, the map $\pi_q : \Kzero(\mathcal{K}^c_{(q)}) \to \Kzero((\mathcal{K}_U)^c_{(q)})$ induced by the Verdier quotient functor is surjective. Now, consider the commutative diagram
	\[
	\xymatrix{
		\Kzero(\mathcal{K}^c_{(p)}) \ar[r]^{i} \ar[d]^-{q^{\natural}} \ar[dr]^{\pi_p} & \Kzero(\mathcal{K}^c_{(p+1)}) \ar[rd]^{\pi_{p+1}} & \\
		\underset{\textstyle = \Cyc^{\Delta}_{p}(\mathcal{T},\mathcal{K})}{\underbrace{\Kzero\left((\mathcal{K}^c_{(p)}/\mathcal{K}^c_{(p-1)})^{\natural}\right)}}  \ar[rd]^{l_p} \ar[rd]^{l_p} & \Kzero((\mathcal{K}_U)^c_{(p)}) \ar[d]^-{q^{\natural}_U} \ar[r]^{i_U} & \Kzero((\mathcal{K}_U)^c_{(p+1)}) \\
		& \underset{\textstyle = \Cyc^{\Delta}_{p}(\mathcal{T}_U, \mathcal{K}_U)}{\underbrace{\Kzero\left(((\mathcal{K}_U)^c_{(p)}/(\mathcal{K}_U)^c_{(p-1)})^{\natural}\right)}} &
	}
	\]
	with $i, i_U$ induced by the inclusion functors and $q^{\natural}, q_U^{\natural}$ induced by the Verdier quotient functors followed by the inclusions into the respective idempotent completions. Since $\pi_p$ is the quotient map killing the subgroup $S$ generated by objects of $\mathcal{K}_Z^c$ and $\pi_{p+1}$ is the quotient map killing $i(S)$, it follows that $\ker(i_U) = \pi_p(\ker(i))$. Thus, the commutativity of the lower-left square yields
	\[q^{\natural}_U(\ker(i_U)) = q^{\natural}_U(\pi_p(\ker(i))) = l_p(q^{\natural}(\ker(i)))~.\]
	In other words, the isomorphism $l_p$ identifies the cycles rationally equivalent to zero in $\Cyc^{\Delta}_{p}(\mathcal{T},\mathcal{K})$ with the cycles rationally equivalent to zero in $\Cyc^{\Delta}_{p}(\mathcal{T}_U, \mathcal{K}_U)$. Hence, the induced map $\ell: \Chow^{\Delta}_p(\mathcal{T}, \mathcal{K}) \to \Chow^{\Delta}_p(\mathcal{T}_U, \mathcal{K}_U)$ is an isomorphism a well.
\end{proof}

\begin{rem}
	For a noetherian scheme $X, \mathcal{T} = \mathrm{D}(X), \mathcal{K} = \mathrm{K}(\mathrm{Inj}(X))$ and the action $\ast$ as defined in Section \ref{sectionkinj}, we have 
	\[\mathcal{K}^c/\mathcal{K}^c_Z \cong \Db(\mathrm{Coh}(X))/\mathrm{D}^{\mathrm{b}}_Z(\mathrm{Coh}(X)) \cong \Db(\mathrm{Coh}(X)/\mathrm{Coh}_Z(X))\]
	where $\Db_Z(\mathrm{Coh}(X)) \subset \Db(\mathrm{Coh}(X))$ denotes the full triangulated subcategory of complexes with homological support contained in $Z$ and $\mathrm{Coh}_Z(X) \subset \mathrm{Coh}(X)$ denotes the Serre subcategory of coherent sheaves with support contained in $Z$. Indeed, this is a consequence of \cite{keller}*{Section 1.15} and \cite{kleinchow}*{Lemma 3.2.5}. Since the bounded derived category of an abelian category is idempotent complete (see \cite{balschlichidem}*{Corollary 2.10}), we see that in this case, the extra condition of Proposition \ref{propchowresiso} is met.
\end{rem}

\small{\textsc{Sebastian Klein, Universiteit Antwerpen, Departement Wiskunde-Informatica, Middelheimcampus, Middelheimlaan 1, 2020 Antwerp, Belgium}\\
	\textit{E-mail address:} \texttt{sebastian.klein@uantwerpen.be}}

\end{document}